\documentclass[11pt,a4paper,leqno,notitlepage]{amsart}
\usepackage[english,frenchb]{babel}
\usepackage{enumerate}
\usepackage{amsmath,todonotes}
\usepackage{amsfonts}
\usepackage{amssymb}
\usepackage{amsthm}
\usepackage{graphicx}
\usepackage{mathrsfs}
\usepackage{datetime}
\usepackage{comment}
\usepackage{fourier-orns}
\usepackage{dsfont}
\usepackage[cyr]{aeguill}
\usepackage{fancyhdr}
\usepackage{mathabx}
\usepackage[latin1]{inputenc}
\usepackage[all,cmtip]{xy}

\usepackage{tikz-cd}
\usetikzlibrary{decorations.pathmorphing}

\usepackage{amsfonts,amsmath,graphicx}

\usepackage{geometry}
\newcommand{\bc}{\begin{center}}
\newcommand{\ec}{\end{center}}
\newcommand{\bq}{\begin{quote}}
\newcommand{\eq}{\end{quote}}
\newcommand{\bqtn}{\begin{quotation}}
\newcommand{\eqtn}{\end{quotation}}
\newcommand{\beq}{\begin{equation}}
\newcommand{\eeq}{\end{equation}}
\newcommand{\bearr}{\begin{eqnarray}}
\newcommand{\eearr}{\end{eqnarray}}
\newcommand{\bearrn}{\begin{eqnarray*}}
\newcommand{\eearrn}{\end{eqnarray*}}
\newcommand{\bi}{\begin{itemize}}
\newcommand{\ei}{\end{itemize}}
\newcommand{\be}{\begin{enumerate}}
\newcommand{\ee}{\end{enumerate}}
\newcommand{\bthe}{\begin{theorem}}
\newcommand{\ethe}{\end{theorem}}
\newcommand{\blem}{\begin{lemme}}
\newcommand{\elem}{\end{lemme}}
\newcommand{\bsolu}{\begin{solution}}
\newcommand{\esolu}{\end{solution}}
\newcommand{\bexer}{\begin{exercise}}
\newcommand{\eexer}{\end{exercise}}

\newcommand{\ba}{\begin{array}}
\newcommand{\ea}{\end{array}}

\usepackage{amsfonts,amsmath}
\newtheorem{theoreme}{Theorem}[section]
\newtheorem{theorem}[theoreme]{Theorem}
\newtheorem{lemme}[theoreme]{Lemma}
\newtheorem{lemma}[theoreme]{Lemma}
\newtheorem{proposition}[theoreme]{Proposition}
\newtheorem{definition}[theoreme]{Definition}

\newtheorem{corollaire}[theoreme]{Corollary}

\newtheorem{solution}[theoreme]{Solution}
\newtheorem{exercise}[theoreme]{Exercise}
\newcommand{\bdefi}{\begin{definition}}
\newcommand{\edefi}{\end{definition}}
\newcommand{\brk}{\begin{remarque}}
\newcommand{\erk}{\end{remarque}}
\newcommand{\bpp}{\begin{proposition}}
\newcommand{\epp}{\end{proposition}}
\newcommand{\bpf}{\begin{proof}}
\newcommand{\epf}{\end{proof}}
\newcommand{\bcor}{\begin{corollaire}}
\newcommand{\ecor}{\end{corollaire}}
\newcommand{\bsol}{\begin{solution}}
\newcommand{\esol}{\end{solution}}
\theoremstyle{definition}

\newtheorem{remarque}[theoreme]{Remark}
\allowdisplaybreaks
\title{Explicit genus formula for any separable cubic global function field}
\author{Sophie Marques and Jacob Ward}

\begin{document}
\large
\selectlanguage{english}
\maketitle
\begin{abstract} 

In the present work, we determine explicitly the genus of any separable cubic extension of any global function field given the minimal polynomial of the extension. We give algorithms computing the ramification data and the genus of any separable cubic extension of any global rational function field. 


\end{abstract}

\noindent \quad {\footnotesize MSC Code (primary): 11T22}

\noindent \quad {\footnotesize MSC Codes (secondary):  11R32, 11R16, 11T55, 11R58}

\tableofcontents
\section*{Introduction}
In \cite{MWcubic3}, we proved that any separable cubic extension of an arbitrary field admits a generator $y$, explicitly determined in terms of an arbitrary initial generating equation, such that
\begin{enumerate}
\item $y^3 =a$, with $a \in F$, or
\item
\begin{enumerate}
\item $y^3 -3y=a$, with $a \in F$, when $p\neq 3$, or
\item $y^3 +ay+a^2 = 0$, with $a \in F$, when $p=3$.
\end{enumerate}
\end{enumerate}
As we will show in this paper, this classification allows one to deduce ramification at any place of $K$ to obtain a very explicit formula for the different, and therefore deduce explicit Riemann-Hurwitz formulae computable entirely using the only parameter of the minimal polynomial for the extension. These data may also be used to calculate integral bases \cite[Theorems 3 and 9]{MadMad}, classify low genus fields, and also have applications to differentials and computation of Weirstrass points. 

In this paper, we describe the ramification for any place in a separable cubic extension of a global field. This study of ramification permits us to obtain in \S 3.3 a Riemann-Hurwitz formula for any separable extension of a cubic global function field (Theorems \ref{RHPC}, \ref{RH}, \ref{char3RH}). In the Appendix, we offer an algorithms which given any irreducible polynomial of degree $3$ return the ramification data and the genus, an explicit integral basis for the given extension over a global rational field.

\section{Notation}
Throughout the paper we denote the characteristic of the field by $p$ (including the possibility $p=0$). 
We let $K$ denote a function field with field of constants $\mathbb{F}_q$, where $q = p^n$ and $p > 0$ is a prime integer. Let $\overline{K}$ be the algebraic closure of $K$. For an extension $L/K$, we let $\mathcal{O}_{L,x}$ denote the integral closure of $\mathbb{F}_q[x]$ in $L$. We denote by $\mathfrak{p}$ a place of $K$ (Section 3). The \emph{degree} $d_K(\mathfrak{p})$ of $\mathfrak{p}$ is defined as the degree of its residue field, which we denote by $k(\mathfrak{p})$, over the constant field $\mathbb{F}_q$. The cardinality of the residue field $k(\mathfrak{p})$ may be written as $|k (\mathfrak{p})|  = q^{d_K(\mathfrak{p})}$. For a place $\mathfrak{P}$ of $L$ over $\mathfrak{p}$, we let $f(\mathfrak{P}|\mathfrak{p}) = [k(\mathfrak{P}):k (\mathfrak{p})]$ denote the inertia degree of $\mathfrak{P}|\mathfrak{p}$. We let $e(\mathfrak{P}|\mathfrak{p})$ be the ramification index of $\mathfrak{P}|\mathfrak{p}$, i.e., the unique positive integer such that $v_\mathfrak{P}(z) = e(\mathfrak{P}|\mathfrak{p})v_\mathfrak{p}(z)$, for all $z \in K$. If $v_\mathfrak{p}(a) \geq 0$, then we let $$\overline{a} := a \mod \mathfrak{p}$$ denote the image of $a$ in $k(\mathfrak{p})$.

Henceforth, we let $F$ denote a field and $p = \text{char}(F)$ the characteristic of this field, where we admit the possibility $p = 0$ unless stated otherwise. We let $\overline{F}$ denote the algebraic closure of $F$. 

\begin{definition} \label{pureclosuredef} \begin{itemize}
\item [$\bullet$] If $p \neq 3$, a generator $y$ of a cubic extension $L/F$ with minimal polynomial of the form $X^3 -a$ $(a \in F)$ is called a {\sf purely} cubic generator, and $L/F$ is called a {\sf purely} cubic extension. If $p = 3$, such an extension is simply called {\sf purely inseparable}.
\item [$\bullet$] If $p\neq 3$ and a cubic extension $L/F$ does not possess a generator with minimal polynomial of this form, then $L/F$ is called {\sf impurely} cubic. 
\item [$\bullet$] For any cubic extension $L/F$, we define the {\sf purely cubic closure} of $L/F$ to be the unique smallest extension $F'$ of $F$ such that $LF'/F'$ is purely cubic. We proved in \cite[Theorem 2.1]{MWcubic3} that the purely cubic closure exists and is unique. 
\end{itemize}
\end{definition}

We would like to point out that by \cite[Corollary 1.2]{MWcubic3}, if $p \neq 3$, then every impure cubic extension $L/K$ has a primitive element $y$ with minimal polynomial of the form $$f(X) = X^3 - 3X - a.$$ We mention this here, as we will use it whenever this case occurs in \S 3. When this is used, the element $y$ will denote any such choice of primitive element.

\section{Function fields} 
\subsection{Constant extensions} 
In this subsection, we wish to determine when a cubic function field over $K$ is a constant extension of $K$. We do this before a study of ramification, as constant extensions are unramified and their splitting behaviour is well understood \cite[Chapter 6]{Vil}. In the subsequent subsections, we will thus assume that our cubic extension $L/K$ is not constant, which, as 3 is prime, is equivalent to assuming that the extension is geometric.

\subsubsection{$X^3-a$, $a\in K$, when $p\neq 3$} 
\begin{lemme}\label{constantextension}
Let $p\neq 3$, and let $L/K$ be purely cubic, i.e. there exists a primitive element $y\in L$ such that $y^3 = a$, $a \in K$. Then $L/K$ is constant if, and only if, $a=ub^3$, where $b\in K$ and $u \in \mathbb{F}_q^*$ is a non-cube. In other words, there is a purely cubic generator $z$ of $L/K$ such that $z^3 = u$, where $u \in \mathbb{F}_q^*$.
\end{lemme} 
\begin{proof}
Suppose that $a=ub^3$, where $b\in K$ and $u$ is a non-cube in $\mathbb{F}_q^*$. Then $z=\frac{y}{b}\in L$ is a generator of $L/K$ such that $z^3 = u$. The polynomial $X^3 -u$ has coefficients in $\mathbb{F}_q$, and as a consequence, $L/K$ is constant. 

Suppose then that $L/K$ is constant. We denote by $l$ the algebraic closure of $\mathbb{F}_q$ in $L$, so that $L=Kl$. Let $l= \mathbb{F}_q (\lambda )$, where $\lambda$ satisfies a cubic polynomial $X^3 + e X^2 + f X + g $ with $e, f , g \in \mathbb{F}_q$. Hence, $L= K(\lambda )$. We denote $$\alpha = -2-\frac{ ( 27 g^2 -9efg +2f^3)^2}{27(3ge-f^2)^3} \in \mathbb{F}_q.$$
As $L/F$ is purely cubic, it then follows by \cite[Corollary 1.2 and Theorem 2.1]{MWcubic3} that either $3eg = f^2$ or the quadratic polynomial $X^2+ \alpha X +1$ has a root in $K$. In both cases, there is a generator $\lambda ' \in L$ such that $$\lambda'^3 = \beta \in \mathbb{F}_q.$$

Hence $\lambda' \in l$. The elements $\lambda'$ and $y$ are two purely cubic generators of $L/K$, whence by \cite[Theorem 3.1]{MWcubic3}, it follows that $y = c \lambda'^j$ where $j=1$, or $2 $ and $c\in K$.  Thus, $a = c^3 \beta^j$, where $\beta\in \mathbb{F}_q$. The result follows.
\end{proof}
\subsubsection{$X^3-3X-a$, $a\in K$, when $p\neq 3$}
Via \cite[Corollary 1.2 and Theorem 3.3]{MWcubic3}, a proof similar to that of Lemma \ref{constantextension} yields the following result. 
\begin{lemme}\label{constantextension2}
Let $p\neq 3$ and  $L/K$ be an impurely cubic extension, so that there is a primitive element $y\in L$ such that $y^3-3y = a$ (see \cite[Corollary 1.2]{MWcubic3}). Then $L/K$ is constant if, and only if, $$u= -3a\alpha^2\beta+a\beta^3+6\alpha+\alpha^3a^2-8\alpha^3 \in \mathbb{F}_q^*,$$ for some $\alpha, \beta \in K$ such that $\alpha^2 + a_2 \alpha \beta + \beta^2 =1$. In other words, there is a generator $z$ of $L/K$ such that $z^3 -3z=u$, where $u \in \mathbb{F}_q^*$.
\end{lemme} 
\subsubsection{$X^3+aX+a^2$, $a\in K$, when $p= 3$} In this case, one may prove the following result, similarly to the proof of Lemma \ref{constantextension}, via \cite[Corollary 1.2 and Theorem 3.6]{MWcubic3}.
\begin{lemme}\label{constantextension2}
Let $p= 3$ and $L/K$ be a separable cubic extension, so that there is a primitive element $y\in L$ such that $y^3+a y+ a^2=0$ (see \cite[Corollary 1.2]{MWcubic3}). Then $L/K$ is constant if, and only if, $$u= \frac{(ja^2 + (w^3 + a w) )^2}{a^3} \in \mathbb{F}_q^*,$$ for some $w\in K$ and $j=1,2$. In other words, there is a generator $z$ of $L/K$ such that $z^3 +uz+u^2=0$, where $u \in \mathbb{F}_q^*$.
\end{lemme} 
\subsection{Ramification} 
In this section, we describe the ramifcation of any place of $K$ in a cubic extension $L/K$. As usual, we divide the analysis into the three fundamental cubic forms derived in \cite[Corollary 1.2]{MWcubic3}.

\subsubsection{$X^3 -a$, $a\in K$, when $p\neq 3$} 
If the extension $L/K$ is purely cubic, one may find a purely cubic generator of a form which is well-suited to a determination of ramification, as in the following lemma.
\begin{lemma}\label{purelocalstandardform}
Let $L/K$ be a purely cubic extension. Given a place $\mathfrak{p}$ of $K$, one may select a primitive element $y$ with minimal polynomial of the form $X^3 -a$ such that either 
\begin{enumerate} 
\item $v_\mathfrak{p}(a)=1,2$, or
\item $v_\mathfrak{p} (a)=0$. 
\end{enumerate}
Such a generator $y$ is said to be in {\sf local standard form} at $\mathfrak{p}$.
\end{lemma} 
\begin{proof} 
Let $y$ be a generator of $L$ such that $y^3=a \in K$. Given a place $\mathfrak{p}$ of $K$, we write $v_\mathfrak{p} (a)= 3 j + r$ with $r=0,1,2$. Via weak approximation, one may find an element $c \in K$ such that $v_\mathfrak{p} (c)= j$. Then $\frac{y}{c}$ is a generator of $L$ such that 
$$ \left(\frac{y}{c}\right)^3 =\frac{y^3}{c^3}= \frac{a}{c^3}$$
and $v_\mathfrak{p} \left( \frac{a}{c^3}\right) = r$. Hence the result.
\end{proof}
When a purely cubic extension $L/K$ is separable, one may also easily determine the fully ramified places in $L/K$.
\begin{theoreme}\label{RPC}
Let $p\neq 3$, and let $L/K$ be a purely cubic extension. Given a purely cubic generator $y$ with minimal polynomial $X^3 -a$, a place $\mathfrak{p}$ of $K$ is ramified if and only if it is fully ramified if, and only if, $(v_\mathfrak{p} (a), 3)=1$.
\end{theoreme} 
\begin{proof} 
Let $\mathfrak{p}$ be a place of $K$ and $\mathfrak{P}$ be a place of $L$ above $\mathfrak{p}$. Suppose that $(v_\mathfrak{p} (a), 3)=1$. Then
$$3v_\mathfrak{P} (y) = v_\mathfrak{P} (y^3)= v_\mathfrak{P} (a)= e(\mathfrak{P}|\mathfrak{p})  v_\mathfrak{p} (a).$$
Since $ (v_\mathfrak{p} (a), 3)=1$, we obtain $ 3 | e(\mathfrak{P}|\mathfrak{p}) $, and as $e(\mathfrak{P}|\mathfrak{p}) \leq 3$, it follows that $e(\mathfrak{P}|\mathfrak{p}) =3$, so that $\mathfrak{p}$ is fully ramified in $L$. 

Conversely, suppose that $(v_\mathfrak{p} (a), 3)\neq 1$. By Lemma \ref{purelocalstandardform}, we know that there exists a generator $z$ of $L$ such that $z^3 -c=0$ and $v_\mathfrak{p} (c)=0$. It is not hard to see that the polynomial $X^3 -a$ is either 
\begin{enumerate} 
\item irreducible modulo $\mathfrak{p}$,
\item $X^3 -a = (X-\alpha ) Q(X) \mod \mathfrak{p}$, where $\alpha \in k(\mathfrak{p})$ and $Q(X)$ is an irreducible quadratic polynomial over modulo $\mathfrak{p}$, or
\item $f(X)  = (X-\alpha ) (X-\beta) (X- \gamma)$ modulo $\mathfrak{p}$ with $\alpha , \beta, \gamma \in k(\mathfrak{p})$ all distinct.
\end{enumerate}
In any of these cases, by Kummer's theorem \cite[Theorem 3.3.7]{Sti}, $\mathfrak{p}$ is either inert or there exist $2$ or $3$ places above it in $L$. Thus, $\mathfrak{p}$ cannot be fully ramified in any case.

Moreover, there are no partially ramified places. Indeed, if $L/K$ is Galois then this is clear and if $L/K$ is not Galois, its Galois closure of $L/K$ is $L(\xi)$ with $K(\xi)/K$ constant therefore unramified and since the index of ramification is multiplicative in tower the only possible index of ramification in $L(\xi)/K$ is $3$ and so is the only possible index of ramification in $L/K$.
\end{proof}

\subsubsection{$X^3-3X-a$, $a\in K$,  $p\neq 3$} 
In order to determine the fully ramified places in extensions of this type, we begin with an elementary but useful lemma. These criteria and notation will be employed throughout what follows.

\begin{lemme}\label{purelycubicclosure}
We consider the polynomial $X^2 +a X +1$ where $a \in K$, we suppose this polynomial is irreducible over $K$, let $c_-,c_+$ denote the roots of this polynomial in $\overline{K}$ the algebraic closure of $K$ and we denote $K(c)$ the quadratic extension $K(c_{\pm})$ of $K$. Then 
$$ c_+ \cdot c_- =1,\ c_+ + c_-= -a, \text{ and } \ \sigma( c_\pm ) = c_\mp,\  \text{ where } \text{\emph{Gal}}(K(c) /K)= \{ Id , \sigma \}.$$ Let $\mathfrak{p}$ be a place of $K$ and $\mathfrak{p}_{c}$ be a place of $K(c )$ above $\mathfrak{p}$. Furthermore, we have:
\begin{enumerate} 
\item For any place $\mathfrak{p}_{c}$ of $K({c})$,
$$v_{\mathfrak{p}_{c}}(c_\pm )= -v_{\mathfrak{p}_{c}}(c_\mp ).$$
\item For any place $\mathfrak{p}_{c}$ of $K({c})$ above a place $\mathfrak{p}$ of $K$ such that $v_{\mathfrak{p}}(a) <0$,
$$ v_{\mathfrak{p}}(a)=-|v_{\mathfrak{p}_{c}}(c_\pm )|,$$ 
and otherwise, $v_{\mathfrak{p}_{c}}(c_\pm )=0$.
\end{enumerate}
\end{lemme} 
\begin{proof} 
\begin{enumerate}
\item At any place $\mathfrak{p}_{c}$ of $K({c})$, we have 
$$v_{\mathfrak{p}_{c}}(c_+\cdot c_-) = v_{\mathfrak{p}_{c}} (c_+)+ v_{\mathfrak{p}_{c}}( c_-)=v_{\mathfrak{p}_{c}}( 1)=0,$$ whence
$$v_{\mathfrak{p}_{c}}(c_+)=-v_{\mathfrak{p}_{c}}( c_-).$$
\item As $c_\pm^2 +a  c_\pm +1 =0,$ the elements $c_\pm' =\frac{ c_\pm }{a}$ satisfy
$$c_\pm'^2 + c_\pm' + \frac{1}{a^2} =0.$$
Thus, for any place $\mathfrak{p}_c$ of $K({c})$ above a place $\mathfrak{p}$ of $K$ such that $v_{\mathfrak{p}}(a) <0$, we obtain
$$v_{\mathfrak{p}_{c}}(c_\pm'^2 + c_\pm' ) = - 2 v_{\mathfrak{p}_{c}}(a)  > 0.$$
By the non-Archimedean triangle inequality, this is possible if, and only if, $v_{\mathfrak{p}_{c}}(c_\pm' ) >0$ or  $v_{\mathfrak{p}_{c}}(c_\pm' ) =0$. If $v_{\mathfrak{p}_{c}}(c_\pm' ) >0$, then
$$v_{\mathfrak{p}_{c}}(c_\pm' )=  - 2 v_{\mathfrak{p}_{c}}(a) \quad \text{ and } \quad v_{\mathfrak{p}_{c}}(c_\pm )=  -  v_{\mathfrak{p}_{c}}(a).$$ 
If on the other hand $v_{\mathfrak{p}_{c}}(c_\pm' ) =0$, we obtain 
$$v_{\mathfrak{p}_{c}}(c_\pm)=   v_{\mathfrak{p}_{c}}(a).$$ 
Thus, the latter together with part $(1)$ of this lemma implies that either 
$$v_{\mathfrak{p}_{c}}(c_+) =  v_{\mathfrak{p}_{c}}(a) \quad \text{and} \quad v_{\mathfrak{p}_{c}}(c_-)=-v_{\mathfrak{p}_{c}}(a)$$ or vice versa (with the roles of $c_-$ and $c_+$ interchanged). 
Moreover, note that $\mathfrak{p}_{c}$ is unramified in $K(c)/ K$ so that $ v_{\mathfrak{p}_{c}}(a)=  v_{\mathfrak{p}}(a)$. For if, when $p\neq 2$, then $K(c)/K$ has a generator $w$ such that $w^2 = -27 (a^2-4)$ and $ 2| v_{\mathfrak{p}}( -27 (a^2-4))$, thus by Kummer theory, $\mathfrak{p}$ is unramified and when $p =2$, $K(c)/K$ has a generator $w$ such that $w^2 -w= \frac{1}{a^2}$ and $  v_{\mathfrak{p}}( \frac{1}{a^2})\geq 0$, thus by Artin-Schreier theory, we have that $\mathfrak{p}$ is unramified in $K(c)$, thus the first part of $(2)$. 
For any place $\mathfrak{p}_{c}$ of $K(b)$ above a place $\mathfrak{p}$ of $K$ such that $v_{\mathfrak{p}}(a) > 0$. As
$$v_{\mathfrak{p}_{c}}(c_\pm'^2 + c_\pm' ) = - 2 v_{\mathfrak{p}_{c}}(a)  < 0,$$
again using the non-Archimedean triangle inequality, we can only have $v_{\mathfrak{p}_{c}}(c_\pm'^2) <0$, whence $v_{\mathfrak{p}_{c}}(c_\pm'^2 ) =  - 2 v_{\mathfrak{p}_{c}}(a).$ 
This implies that $v_{\mathfrak{p}_{c}}(c_\pm ) =0$. Finally, via the triangle inequality once more, for any $\mathfrak{p}_{c}$ such that $v_{\mathfrak{p}_{c}}(a) = 0$, we must have
$v_{\mathfrak{p}_{c}}(c_\pm ) =0.$
\end{enumerate}
\end{proof}

\begin{theorem} \label{ramification}
 Let $p \neq 3$, and let $L/K$ be an impurely cubic extension and $y$ primitive element with minimal polynomial $f(X) = X^3-3X-a$. Then 
 \begin{enumerate}
 \item the fully ramified places of $K$ in $L$ are precisely those $\mathfrak{p}$ such that $(v_{\mathfrak{p}}(a), 3)=1$ and 
 \item the partially ramified places $\mathfrak{p}$ that is the one with index of ramification $2$ are precisely such that $a \equiv \pm 2 \mod \mathfrak{p}$ and 
 \begin{enumerate}
 \item  $(v_{\mathfrak{p}} ( a^2 - 4) , 2)=1 $, when $p \neq 2$;
\item  there exist $w \in K$ such that $v_{\mathfrak{p}} ( 1/a +1 + w^2 -w ) <0$ and $(v_{\mathfrak{p}} (1/a+1  + w^2 -w) , 2)=1$, when $p=2$.
 \end{enumerate} 
 \end{enumerate}
\end{theorem} 

\begin{proof} 
\begin{enumerate}
\item As usual, we let $\xi$ be a primitive $3^{rd}$ root of unity. We also let $r$ be a root of the quadratic resolvent $R(X)= X^2 +3aX + ( -27 + 9 a^2)$ of the cubic polynomial $X^3-3X-a$ in $\overline{K}$. As in \cite[Theorem 2.3]{Con}, we know that $L(r)/K(r)$ is Galois, and by \cite[Corollary 1.6]{MWcubic3}, we have that $L(\xi, r) / K(\xi , r)$ is purely cubic. We denote by $\mathfrak{p}$ a place in $K$, $\mathfrak{P}_{\xi,r}$ a place of $L(\xi,r)$ above $\mathfrak{p}$, $\mathfrak{P}=\mathfrak{P}_{\xi,r}\cap L$, and $\mathfrak{p}_{\xi,r}=\mathfrak{P}_{\xi,r}\cap K(\xi , r)$. By \cite[Theorem 1.5]{MWcubic3}, we know that that $L(\xi, r)/K(\xi, r)$ is Kummer; more precisely, there exists $v \in K(\xi , r)$ such that $v^3 = c$ where $c$ is a root of the polynomial $X^2+aX+1$. 

We thus obtain a tower $L(\xi, r) / K(\xi, r) /K(\xi) / K$ with $L(\xi, r )/K(\xi,r )$ Kummer of degree $3$, and where $K(\xi ,r)/ K(\xi)$ and $K(\xi)/K$ are both Kummer extensions of degree $2$. As the index of ramification is multiplicative in towers and the degree of $L(\xi , r) / K(\xi ,r)$ and $K(\xi ,r)/K$ are coprime, the places of $K$ that fully ramify in $L$ are those places of $K$ which lie below those of $K(\xi , r )$ which fully ramify in $L(\xi, r)/K(\xi, r)$. As $L(\xi,r)/K(\xi,r)$ is Kummer, the places of $K(\xi ,r)$ that ramify in $L(\xi,r)$ are described precisely by Kummer theory (see for example \cite[Example 5.8.9]{Vil}) as those $\mathfrak{p}_{\xi,r}$ in $K(\xi ,r)$ such that $$(v_{\mathfrak{p}_{\xi,r}} (c_\pm ) , 3)=1.$$ 
Lemma \ref{purelycubicclosure} states that if $v_{\mathfrak{p}}(a) <0$, then $v_{\mathfrak{p}_{\xi, r}}(c_\pm )= \pm v_{\mathfrak{p}_{\xi,r}}(a)$ and that otherwise, $v_{\mathfrak{p}_{\xi,r}}(c_\pm )=0$. Thus, the ramified places of $L/F$ are those places $\mathfrak{p}$ below a place $\mathfrak{p}_{\xi , r}$ of $K(\xi ,r)$ such that $(v_{\mathfrak{p}_{\xi,r}}(a), 3)=1$. Also, $$v_{\mathfrak{p}_{\xi, r}}(a) = e(\mathfrak{p}_{\xi, r} | \mathfrak{p}) v_{\mathfrak{p}}(a),$$ where $e(\mathfrak{p}_{\xi, r} | \mathfrak{p})$ is the ramification index of a place $\mathfrak{p}$ of $K$ in $K(\xi, r)$, equal to $1$, $2$, or $4$, and in any case, coprime with $3$. Thus, $(v_{\mathfrak{p}_{\xi,r}}(a),3)=1$ if, and only if, $(v_{\mathfrak{p}}(a),3)=1$. As a consequence of the above argument, it therefore follows that a place $\mathfrak{p}$ of $K$ is fully ramified in $L$ if, and only if, $v_{\mathfrak{p}}(a)<0$.
If $L/K$ is Galois then all the places are fully ramified. 

\item Now, if $L/K$ is not Galois and a ramified place $\mathfrak{p}$ is not fully ramified in $L/K$ its index of ramification is $2$. 
 The Galois closure of $L/K$ is $L(r)/K$ since $L(r)/K(r)$ is Galois then all the ramified places in $L(r)/ K(r)$ are fully ramified and the only possible way that the index of ramification of a place is $2$ in $L/K$ is that this place is ramified in $K(r)/K$ since the index of ramification is multiplicative in tower. By Kummer and Artin-Schreier theory this implies that $v_{\mathfrak{p}}(a)>0$. Since $K(r)/K$ is defined by a minimal equation $X^2 = -27 (a^2-4)$ when $p \neq 2$ and and $X^2 -X = 1 + 1/a$ when $p=2$. 


When $v_{\mathfrak{p}} (a)\geq 0$, via Kummer's theorem, for $\mathfrak{p}$ to be partially ramified in $L/K$, the only possible decomposition of $X^3 -3 X -a$ mod $\mathfrak{p}$ is  
$$X^3 -3 X -a = (X-\alpha)^2 ( X-\beta) \mod \mathfrak{p}$$

The equality $f(X) = (X - \alpha)  (X - \beta)^2$ gives us $$X^3 - 3X - a = (X - \alpha)(X - \beta)^2 = X^3 - (2\beta + \alpha) X^2 + (\beta^2 + 2\alpha \beta) X - \alpha \beta^2.$$ Thus $\alpha = - 2\beta$.
We therefore have $- 3 = \beta^2 - 4\beta^2 = - 3 \beta^2$ and $a = - 2\beta^3$. The first of these implies that $$3(\beta^2 - 1) = 3 \beta^2 - 3 = 0.$$ Thus $\beta = \pm 1$ and $a = \mp 2$. Conversely, when $a = \mp 2$, then 
$$X^3-3 X \mp 2= ( X \pm 2)(X\mp 1 )^2.$$

Therefore, in order for $\mathfrak{p}$ to be partially ramified we need that $\mathfrak{p}$ ramified in $K(r)$ that is  
 \begin{enumerate}
 \item  $(v_{\mathfrak{p}} ( a^2 - 4) , 2)=1 $, when $p \neq 2$;
\item  there exist $w \in K$ such that $v_{\mathfrak{p}} ( 1/a +1 + w^2 -w ) <0$ and $(v_{\mathfrak{p}} (1/a+1  + w^2 -w) , 2)=1$, when $p=2$.
 \end{enumerate} 
 and $a \equiv \mp 2  \mod \ \mathfrak{p}$.

Conversely, suppose that $\mathfrak{p}$ is a place such that $a \equiv \mp 2  \mod \ \mathfrak{p}$ and $\mathfrak{p}$ ramified in $K(r)$.  Since $\mathfrak{p}$ cannot be ramified in $K(r)$ without $v_{\mathfrak{p}}(a) \geq 0$ and $L(r)/K$ is Galois, then when $a \equiv \mp 2  \mod \ \mathfrak{p}$ and $\mathfrak{p}$ ramified in $K(r)$, then the place above $\mathfrak{p}$ in $K(r)$ is unramified in $L(r)/K(r)$ (see proof of (1)) therefore completely split and we must have $$\mathfrak{p} \mathcal{O}_{L(r),x} = (\mathfrak{P}_{1,r}\mathfrak{P}_{2,r} \mathfrak{P}_{3,r})^2$$ 
 
Since  $a \equiv \mp 2  \mod \ \mathfrak{p}$, we have $X^3 - 3 X -a = (X-\alpha ) (X-\beta)^2 \mod \mathfrak{p} $ with $\alpha , \beta \in  k (\mathfrak{p})$ and $\alpha \neq \beta$, by Kummer's theorem, we know that there is at least two place above $\mathfrak{p}$ in $L$ thus either
\begin{enumerate} 
\item $\mathfrak{p}\mathcal{O}_{L,x} = \mathfrak{P}_1 \mathfrak{P}_2$ where $\mathfrak{P}_i$, $i=1,2$ place of $L$ above $\mathfrak{p}$, or 
\item  $\mathfrak{p}\mathcal{O}_{L,x} = \mathfrak{P}_1^2 \mathfrak{P}_2$ where $\mathfrak{P}_i$, $i=1,2$ place of $L$ above $\mathfrak{p}$, or
\item $\mathfrak{p}\mathcal{O}_{L,x} = \mathfrak{P}_1 \mathfrak{P}_2 \mathfrak{P}_3$ where $\mathfrak{P}_i$, $i=1,2,3$ place of $L$ above $\mathfrak{p}$.
\end{enumerate}
By \cite[p. 55]{Neu}, we know that $\mathfrak{p}$ is completely split in $L$ (case (c)) if, and only if, (1) $\mathfrak{p}$ is completely split in $K(r)$ and (2) $\mathfrak{p}_r$ completely split in $L(r)$.  Thus, either $\mathfrak{p}\mathcal{O}_{L,x} = \mathfrak{P}_1 \mathfrak{P}_2$ or $\mathfrak{p}\mathcal{O}_{L,x} = \mathfrak{P}_1^2 \mathfrak{P}_2$ where each $\mathfrak{P}_i$ ($i=1,2$) is a place of $L$ above $\mathfrak{p}$. Note that $2\;|\; e (  \mathfrak{P}_{r}| \mathfrak{p})$  for any places $\mathfrak{P}_r$ in $L(r)$ above $\mathfrak{p}$.  
If $\mathfrak{p}\mathcal{O}_{L,x} = \mathfrak{P}_1 \mathfrak{P}_2$, then as $e( \mathfrak{P}_i| \mathfrak{p} )=1$, we have that $2\;|\;e( \mathfrak{P}_{i,r}| \mathfrak{P}_i)$
and $\mathfrak{p}\mathcal{O}_{L(r),x} = \mathfrak{P}_{1,r}^2 \mathfrak{P}_{2,r}^2$, where $\mathfrak{P}_{i,r}$, $i=1,2$ are places above $\mathfrak{p}$ in $L(r)$, which is impossible, as $\mathfrak{p} \mathcal{O}_{L(r),x} = (\mathfrak{P}_{1,r}\mathfrak{P}_{2,r} \mathfrak{P}_{3,r})^2$. Thus, in this case, we must have $\mathfrak{p}\mathcal{O}_{L,x} = \mathfrak{P}_1^2 \mathfrak{P}_2$ and $\mathfrak{P}_1$ is split in $K(r)$ and $\mathfrak{P}_2$ ramifies in $K(r)$.
\end{enumerate}

\end{proof}

This theorem yields the following corollaries, the first being immediate.

\begin{corollaire} \label{ramificationgalois}
Suppose that $q \equiv - 1 \mod 3$. Let $L/K$ be a Galois cubic extension, so that there exists a primitive element $y$ of $L$ with minimal polynomial $f(X) = X^3-3X-a$. Then the (fully) ramified places of $K$ in $L$ are precisely those places $\mathfrak{p}$ of $K$ such that $v_{\mathfrak{p}}(a)<0$ and $(v_{\mathfrak{p}}(a), 3)=1$. 
\end{corollaire} 
\begin{corollaire} \label{odddegree}
Suppose that $q \equiv - 1 \mod 3$. Let $L/K$ be a Galois cubic extension, so that there exists a primitive element $y$  of $L$ with minimal polynomial $f(X) = X^3-3X-a$. Then, only those places of $K$ of even degree can (fully) ramify in $L$. More precisely, any place $\mathfrak{p}$ of $K$ such that $v_\mathfrak{p}(a) <0$ is of even degree. 
\end{corollaire}
\begin{proof} 
In Lemma \ref{purelycubicclosure}, it was noted that $ \sigma( c_\pm ) = c_\mp$ where $\text{\emph{Gal}}(K(c_\pm ) /K)= \{ Id , \sigma \}$, when $c_\pm \notin K$. Let $\xi$ again be a primitive $3^{rd}$ root of unity. We denote by $\mathfrak{p}$ a place of $K$ and $\mathfrak{p}_\xi$ a place of $K(\xi)$ above $\mathfrak{p}$. We find that
$$v_{\mathfrak{p}_\xi}(c_\pm) = v_{\sigma(\mathfrak{p}_\xi)}(\sigma ( c_\pm ))=v_{\sigma(\mathfrak{p}_\xi)}( c_\mp ).$$
Note that if $\sigma (  \mathfrak{p}_\xi )= \mathfrak{p}_\xi$, it follows that  $v_{\mathfrak{p}_\xi}(c_\pm) =v_{\mathfrak{p}_\xi}( c_\mp )$. However, by Lemma \ref{purelycubicclosure}, we have that, for any place $\mathfrak{p}_\xi$ of $K(\xi)$ above a place $\mathfrak{p}$ of $K$ such that $v_{\mathfrak{p}}(a) <0$, it holds that $v_{\mathfrak{p}_\xi}(c_\pm )= \pm v_{\mathfrak{p}_\xi}(a)$, 
and that $v_{\mathfrak{p}_\xi }(c_\pm )= -v_{\mathfrak{p}_\xi }(c_\mp )$. Thus, for any place $\mathfrak{p}_\xi$ of $K(\xi )$ above a place $\mathfrak{p}$ of $K$ such that $v_{\mathfrak{p}}(a) <0$, we find that $v_{\mathfrak{p}_\xi}(c_\pm )\neq v_{\mathfrak{p}_\xi}(c_\mp )$ and thus $\sigma (  \mathfrak{p}_\xi )\neq \mathfrak{p}_\xi$.
Therefore, by \cite[Theorem 6.2.1]{Vil}, we obtain that $ \mathfrak{p} $ is of even degree, for any place $\mathfrak{p}$ of $K$ such that $v_{\mathfrak{p}}(a)<0$.
\end{proof}
\begin{corollaire} \label{infty}
Suppose that $q \equiv - 1 \mod 3$. Let $L/K$ be a Galois cubic extension, so that there exists a primitive element $y$ of $L$ with minimal polynomial $f(X)=X^3-3 X -a$. Then one can choose a single place $\mathfrak{P}_\infty$ at infinity in $K$ such that $v_{\mathfrak{P}_\infty}(a) \geq 0$. 
\end{corollaire} 
\begin{proof} 

One can choose $x \in K \backslash \mathbb{F}_q$ such that the place $\mathfrak{p}_\infty$ at infinity for $x$ has the property that all of the places in $K$ above it are of odd degree. In order to accomplish this, we appeal to a method similar to the proof of \cite[Proposition 7.2.6]{Vil}; because there exists a divisor of degree 1 \cite[Theorem 6.3.8]{Vil}, there exists a prime divisor $\mathfrak{P}_\infty$ of $K$ of odd degree; for if all prime divisors of $K$ were of even degree, then the image of the degree function of $K$ would lie in $2\mathbb{Z}$, which contradicts \cite[Theorem 6.3.8]{Vil}. Let $d$ be this degree. Let $m \in \mathbb{N}$ be such that $m > 2 g_K - 1$. Then, by the Riemann-Roch theorem \cite[Corollary 3.5.8]{Vil}, it follows that there exists $x \in K$ such that the pole divisor of $x$ in $K$ is equal to $\mathfrak{P}_\infty^m$. By definition, the pole divisor of $x$ in $k(x)$ is equal to $\mathfrak{p}_\infty$. It follows that $$(\mathfrak{p}_\infty)_K = \mathfrak{P}_\infty^m,$$ from which it follows that $\mathfrak{P}_\infty$ is the unique place of $K$ above $\mathfrak{p}_\infty$, and by supposition that $\mathfrak{P}_\infty$ is of odd degree. From this argument, we obtain that, with this choice of infinity, all places above infinity in $k(x)$ are of odd degree. (We also note that we may very well choose $m$ relatively prime to $p$, whence $K/k(x)$ is also separable; in general, $K/k(x)$ as chosen will not be Galois.) 

As $q \equiv -1 \mod 3$, $L/K$ is a Galois extension, and $y$ is a primitive element with minimal polynomial of the form $X^3 -3 X -a$ where $a \in K$, we know that all of the places $\mathfrak{p}$ of $K$ such that $v_\mathfrak{p} (a)<0$, and in particular, all the ramified places, are of even degree (see Corollary \ref{odddegree}). It follows that the process described in this proof gives the desired construction, and the result follows.
\end{proof}
\begin{remarque}
We note that when $K$ is a rational function field, one may use Corollary \ref{infty} to show that the parameter $a$ has nonnegative valuation at $\mathfrak{p}_\infty$ for a choice of $x$ such that $K = \mathbb{F}_q(x)$, and thus such $\mathfrak{p}_\infty$ is unramified.
\end{remarque}

\subsubsection{$X^3+aX+a^2$, $a\in K$, $p = 3$} 
As for purely cubic extensions, there exist a local standard form which is useful for a study of splitting and ramification.
\begin{lemma}\label{char3localstandardform1}
Let $p=3$, and let $L/K$ be a cubic separable extension. Let $\mathfrak{p}$ be a place of $K$. Then there is a generator $y$ such that $y^3 +a y +a ^2 =0$ such that $v_\mathfrak{p}(a) \geq 0$, or 
$v_\mathfrak{p} (a) <0 $ and $( v_\mathfrak{p} (a), 3)=1$. Such a $y$ is said to be {\sf in local standard form} at $\mathfrak{p}$.
\end{lemma}
\begin{proof} 
Let $\mathfrak{p}$ be a place of $K$. Let $y_1$ be a generator of $L/K$ such that $y_1^3 +a_1 y_1 + a_1^2=0$ (this was shown to exist in \cite{MWcubic}). By \cite[Theorem 3.6]{MWcubic3}, any other generator $y_2$ with a minimal equation of the same form $y_2^3 +a_2 y_2 +a_2^2=0$ is such that $y_2 =-\beta (\frac{j}{a_1}y_1- \frac{1}{a_1} w )$, and we have $$a_2  = \frac{(ja_1^2 + (w^3 + a_1 w) )^2}{a_1^3}.$$ Suppose that $v_\mathfrak{p}(a_1)  < 0$, and that $3 \;|\; v_\mathfrak{p}(a_1)$. Using the weak approximation theorem, we choose $\alpha \in K$ such that $v_\mathfrak{p}(\alpha) = 2v_\mathfrak{p}(a_1)/3$, which exists as $3 \;|\; v_\mathfrak{p}(a_1)$. Then $$v_\mathfrak{p}(\alpha^{-3} ja_1^2) = 0.$$ Let $w_0 \in K$ be chosen so that $w_0 \neq - \alpha^{-3} ja_1^2$ and $$v_\mathfrak{p}(\alpha^{-3} ja_1^2 + w_0) > 0.$$ 
This may be done via the following simple argument: As $v_\mathfrak{p}(\alpha^{-3} ja_1^2) = 0$, then $\overline{\alpha^{-3} ja_1^2} \neq 0$ in $k(\mathfrak{p})$. 

We then choose some $w_0 \neq - \alpha^{-3} ja_1^2 \in K$ such that $\overline{ w_0} = -\overline{\alpha^{-3} ja_1^2}$ in $k(\mathfrak{p})$. Note that $v_{\mathfrak{p}}(w_0)=0$. Thus, $\overline{\alpha^{-3} ja_1^2 + w_0}=0$ in $k(\mathfrak{p})$ and $v_\mathfrak{p}(\alpha^{-3} ja_1^2 + w_0) > 0$. As $p = 3$, it follows that the map $X \rightarrow X^3$ is an isomorphism of $k(\mathfrak{p})$, so we may find an element $w_1 \in K$ such that $w_1^3 = w_0 \mod \mathfrak{p}$. Hence $$v_\mathfrak{p}(\alpha^{-3} ja_1^2 + w_1^3) > 0.$$ We then let $w_2 = \alpha w_1$, so that $$v_\mathfrak{p}(j a_1^2 + w_2^3) = v_\mathfrak{p}(j a_1^2 + \alpha^3 w_1^3) > v_\mathfrak{p}(j a_1^2) .$$ Thus, as $v_\mathfrak{p}(a_1) < 0$, we obtain \begin{align*} v_\mathfrak{p}(ja_1^2 + (w_2^3 + a_1 w_2)) &\geq \min\{v_\mathfrak{p}(ja_1^2 + w_2^3  ) ,v_\mathfrak{p}(a_1 w_2) \} \\&> \min\{v_\mathfrak{p}(j a_1^2),v_\mathfrak{p}(a_1 w_2)\}\\& = \min\{v_\mathfrak{p}(j a_1^2),v_\mathfrak{p}(a_1) + 2v_\mathfrak{p}(a_1)/3\} \\& =  \min\{2v_\mathfrak{p}(a_1),5v_\mathfrak{p}(a_1)/3\} \\& = 2v_\mathfrak{p}(a_1). \end{align*} Hence $$v_\mathfrak{p}(a_2)  = v_\mathfrak{p}\left(\frac{(ja_1^2 + (w^3 + a_1 w) )^2}{a_1^3}\right) > 4v_\mathfrak{p}(a_1) - v_\mathfrak{p}(a_1^3) = v_\mathfrak{p}(a_1).$$ We can thus ensure (after possibly repeating this process if needed) that we terminate at an element $a_2 \in K$ for which $v_\mathfrak{p}(a_2) \geq 0$ or for which $v_\mathfrak{p}(a_2) < 0$ and $(v_\mathfrak{p}(a_2),3) = 1$. 
\end{proof} 

\begin{remarque}
Note that we can do what we have done in the previous Lemma simultaneously at any finite place (see \cite[Lemma 1.2]{MWcubic4}).
\end{remarque} 

\begin{theorem} \label{char3localstandardform}
Suppose that $p = 3$. Let $L/K$ be a separable cubic extension and $y$ a primitive element with minimal polynomial $X^3 + aX + a^2$. Let $\mathfrak{p}$ be a place of $K$ and $\mathfrak{P}$ a place of $L$ above $\mathfrak{p}$. Then 
\begin{enumerate} 
\item $\mathfrak{p}$ is fully ramified if, and only if, there is $w \in K$, $v_\mathfrak{p}(\alpha ) < 0$ and $(v_\mathfrak{p}(\alpha ),3) = 1$ with $$\alpha = \frac{(ja^2 + (w^3 + a w) )^2}{a^3}.$$ Equivalently, there is a generator $z$ of $L$ whose minimal polynomial is of the form $X^3 + \alpha X + \alpha^2$, where $v_\mathfrak{p}(\alpha ) < 0$ and $(v_\mathfrak{p}(\alpha ),3) = 1$, and 
\item $\mathfrak{p}$ is partially ramified if and only if $(v_\mathfrak{p}( a) , 2)=1$ and  there is $w \in K$ such that $v_\mathfrak{p}(\alpha ) \geq 0$ with $$\alpha = \frac{(ja^2 + (w^3 + a w) )^2}{a^3}.$$ The later is equivalent to the existence of a generator $z$ of $L$ whose minimal polynomial is of the form $X^3 + \alpha X + \alpha^2$, where $v_\mathfrak{p}(\alpha ) \geq 0$.
\end{enumerate} 
\end{theorem} 
\begin{proof} 
\begin{enumerate} 
\item Let $\mathfrak{p}$ be a place of $K$, and denote by $\mathfrak{P}$ a place of $L$ above $\mathfrak{p}$. When $L/F$ is Galois, this theorem is simply the usual Artin-Schreier theory (see \cite[Proposition 3.7.8]{Sti}). Otherwise, since the discriminant of the polynomial $X^3+a X+a^2$ is equal to $\Delta=-4a^3=-a^3$, by \cite[Theorem 2.3]{Con}, we know that the Galois closure of $L/F$ is equal to $L(\Delta )= L(b)$, where $b^2 = -a$. Let $\mathfrak{p}_b$ a place of $K(b)$ above $\mathfrak{p}$. The extension $L(b)/K(b)$ is an Artin-Schreier extension with Artin-Schreier generator $y/b$ possessing minimal polynomial $X^3-X+b$. 
As $L(b)/K(b)$ is Galois, if $\mathfrak{p}_b$ is ramified in $L(b)$, then it must be fully ramified. Furthermore, as the degree $K(b)/K$ is equal to $2$, which is coprime with $3$, and the index of ramification is multiplicative in towers, it follows that the place $\mathfrak{p}$ is fully ramified in $L$ if, and only if, $\mathfrak{p}_b$ is fully ramified in $L(b)$. 
By \cite[Proposition 3.7.8]{Sti}, \begin{enumerate} \item $\mathfrak{p}_b$ is fully ramified in $L(b)$ if, and only if, there is an Artin-Schreier generator $z$ such that $z^3-z -c$ with $v_{\mathfrak{p}_b}(c) < 0$ and $(v_{\mathfrak{p}_b}(c),3)=1$, and \item $\mathfrak{p}_b$ is unramified in $L(b)$ if, and only if, there is an Artin-Schreier generator $z$ such that $z^3-z -c$ with $v_{\mathfrak{p}_b}(c) \geq 0$. \end{enumerate}
Suppose that there is a generator $w$ such that $w^3 +a_1 w +a_1 ^2 =0$, $v_\mathfrak{p} (a_1) <0 $ and $( v_\mathfrak{p} (a_1), 3)=1$. Then over $K(b_1)$, where $b_1^2 =-a_1$, we have an Artin-Schreier generator $z$ of $L(b_1)$ such that  $z^3-z +b_1$. Moreover, $$v_{\mathfrak{p}_{b_1}}(b_1)= \frac{ v_{\mathfrak{p}_{b_1}}(a_1)}{2}= \frac{e(\mathfrak{p}_{b_1} |\mathfrak{p}) v_{\mathfrak{p}}(a_1)}{2},$$ 
where $e(\mathfrak{p}_{b_1} |\mathfrak{p})$ is the index of ramification of $\mathfrak{p}_{b_1}$ over $K(b_1)$, whence $e(\mathfrak{p}_{b_1} |\mathfrak{p})=1$ or $2$. 
As a consequence, $$(v_{\mathfrak{p}_{b_1}}(b_1), 3)=( v_\mathfrak{p} (a_1), 3)=1,$$ and $\mathfrak{p}_{b_1}$ is fully ramified in $L(b_1)$, so that $\mathfrak{p}$ too must be fully ramified in $L$. 

Suppose that there exists a generator $w$ such that $w^3 +a_1 w +a_1 ^2 =0$, $v_\mathfrak{p} (a_1) \geq 0$. Then over $K(b_1)$, where $b_1^2 =-a_1$, we have a generator $z$ of $L(b_1)$ such that  $z^3-z +b_1$ and $$v_{\mathfrak{p}_{b_1}}(b_1)= \frac{e(\mathfrak{p}_{b_1} |\mathfrak{p}) v_{\mathfrak{p}}(a_1)}{2}\geq 0.$$ 
Thus $\mathfrak{p}_b$ is unramified in $L(b)$, so that $\mathfrak{p}$ cannot be fully ramified in $L$, since the ramification index is multiplicative in towers. The theorem then follows by Lemma \ref{char3localstandardform1}. If $L/K$ is Galois, then the ramified places are all fully ramified.
\item If $L/K$ is not Galois and a ramified place $\mathfrak{p}$ is not fully ramified in $L/K$ its index of ramification is $2$. Moreover, when $\mathfrak{p}$ is not fully ramified we know by $(1)$ and Lemma \ref{char3localstandardform1} that there is $w \in K$ such that $v_\mathfrak{p}(\alpha ) \geq 0$ with $$\alpha = \frac{(ja^2 + (w^3 + a w) )^2}{a^3}.$$ 
 The Galois closure of $L/K$ is $L(b)/K$ where $b^2 = -\alpha$ since $L(b)/K(b)$ is Galois then all the ramified places in $L(b)/ K(b)$ are fully ramified and the only possible way that the index of ramification of a place is $2$ in $L/K$ is that this place is ramified in $K(b)/K$ since the index of ramification is multiplicative in tower. 
That is $$(v_{\mathfrak{p}}(a),2)= (v_{\mathfrak{p}}(\alpha) , 2)=1.$$ 
Since the Galois closure has also as generator $c$ such that $c^2 = -a$. Therefore, $v_{\mathfrak{p}} (\alpha )>0$ and $(v_{\mathfrak{p}}(a),2)=1$.


Conversely, suppose there is $w \in K$ such that $v_\mathfrak{p}(\alpha ) > 0$ with $$\alpha = \frac{(ja^2 + (w^3 + a w) )^2}{a^3}$$ and  $(v_{\mathfrak{p}}(a),2)=1$.
If $v_\mathfrak{p}(\alpha) > 0$, then $L(b)/K(b)$ is an Artin-Schreier extension by \cite[Theorem 2.3]{Con}, and there is an Artin-Schreier generator $w=\frac{z}{b}$ such that $w^3 - w + b=0$ and $v_{\mathfrak{p}_b} ( b)>0$, where $\mathfrak{p}_b$ is a place of $K(d)$ above $\mathfrak{p}$. Thus $b \equiv 0 \mod \mathfrak{p}_b$, and the polynomial $$X^3 -X +b \equiv X^3 - X \mod \mathfrak{p}_b$$ factors as $X(X-1)(X+1)$ modulo $\mathfrak{p}_b$. By Kummer's theorem (\cite[Theorem 3.3.7]{Sti}), we then have that $\mathfrak{p}_b$ is completely split in $L(b)$. 

As $\mathfrak{p}_b$ is completely split in $L(b )$, we have that $\mathfrak{p}$ cannot be inert in $L$. Indeed, if $\mathfrak{p}$ were inert in $L$, then there are at most two places above $\mathfrak{p}$ in $L(b)$, in contradiction with the proven fact that $\mathfrak{p}_b$ is completely split in $L(b)$. 

By \cite[p.55]{Neu}, $\mathfrak{p}$ splits completely in $L$ if, and only if, $\mathfrak{p}$ is completely split in $K(b)$ and $\mathfrak{p}_b$ is completely split in $L(b)$.

Also, since by the previous argument $\mathfrak{p}$ cannot be inert in $L$, we have that either $$\mathfrak{p}\mathcal{O}_{L,x} = \mathfrak{P}_1 \mathfrak{P}_2\qquad\text{or}\qquad\mathfrak{p}\mathcal{O}_{L,x} = \mathfrak{P}_1 \mathfrak{P}_2^2,$$ where $\mathfrak{P}_i$, $i=1,2$ are places of $L$ above $\mathfrak{p}$. Let $\mathfrak{P}_b$ be a place of $L(b)$ above $\mathfrak{p}$. When $\mathfrak{p}$ is ramified in $K(b)$, then the index of ramification at any place above $\mathfrak{p}$ in $L(b)$ is divisible by $2$, since $L(b)/K$ is Galois by \cite[Theorem 2.3]{Con}, whence $\mathfrak{p}\mathcal{O}_{L,x} = \mathfrak{P}_1 \mathfrak{P}_2^2$.

\end{enumerate} 
\end{proof} 

\subsection{Riemann-Hurwitz formulae}
Using the extension data, it is possible to give the Riemann-Hurwitz theorem for each of our forms in \cite[Corollary 1.2]{MWcubic3}. These depend only on information from a single parameter.
\subsection{$X^3 -a$, $a\in K$, $p \neq 3$} 
\begin{lemma} 
 \label{diffexpqneq1mod3}
Let $p \neq 3$. Let $L/K$ be a purely cubic extension and $y$ a primitive element of $L$ with minimal polynomial $f(X) = X^3-a$. Let $\mathfrak{p}$ be a place of $K$ and $\mathfrak{P}$ a place of $L$ over $\mathfrak{p}$. Then the following are true:
\begin{enumerate} 
\item $d(\mathfrak{P}|\mathfrak{p}) =0$ if, and only if, $e(\mathfrak{P}|\mathfrak{p})=1$. 
\item  $d(\mathfrak{P}|\mathfrak{p}) =2$, otherwise. That is, $e(\mathfrak{P}|\mathfrak{p})=3$, which by Theorem \ref{RPC} is equivalent to $(v_{\mathfrak{p}}(a), 3)=1$.
\end{enumerate} 
\end{lemma} 
\begin{proof}
By Theorem \ref{RPC}, either $e(\mathfrak{P}|\mathfrak{p})=1$ or $e(\mathfrak{P}|\mathfrak{p})=3$.
 \begin{enumerate} 
\item  As the constant field $\mathbb{F}_q$ of $K$ is perfect, all residue field extensions in $L/K$ are automatically separable. The result then follows from \cite[Theorem 5.6.3]{Vil}.
\item If $e(\mathfrak{P}|\mathfrak{p})=3$, then as $p \nmid 3$, it follows again from [Theorem 5.6.3, Ibid.] that $d(\mathfrak{P}|\mathfrak{p}) = e(\mathfrak{P}|\mathfrak{p}) - 1 =   2$.
\end{enumerate}
\end{proof} 
We thus find the Riemann-Hurwitz formula as follows for purely cubic extensions when the characteristic is not equal to 3, which resembles that of Kummer extensions, but no assumption is made that the extension is Galois.
\begin{theoreme}[Riemann-Hurwitz I]  \label{RHPC}  Let $p \neq 3$. Let $L/K$ be a purely cubic  geometric extension, and $y$ a primitive element of $L$ with minimal polynomial $f(X) = X^3-a$. Then the genus $g_L$ of $L$ is given according to the formula $$g_L =  3g_K - 2 +  \sum_{\substack{  ( v_\mathfrak{p}(a),3) = 1}} d_{K}(\mathfrak{p}).$$
\end{theoreme}
\begin{proof}
This follows from Lemma \ref{diffexpqneq1mod3}, \cite[Theorem 9.4.2]{Vil}, and the fundamental identity $\sum e_i f_i = [L:K] = 3$. 
\end{proof} 
\subsection{$X^3 -3X -a$, $a\in K$, $p\neq 3$} 
\begin{lemma} 
 \label{diffexpqneq1mod3}
 Let $p \neq 3$. Let $L/K$ be an impurely cubic extension and $y$ a primitive element of $L$ with minimal polynomial $f(X) = X^3-3X-a$. Let $\mathfrak{p}$ be a place of $K$ and $\mathfrak{P}$ a place of $L$ over $\mathfrak{p}$. Let $\Delta=-27(a^2-4)$ be the discriminant of $f(X)$ and $r \in \overline{K}$ a root of the quadratic resolvent $R(X)= X^2 +3aX + ( -27 + 9 a^2)$ of $f(X)$. Then the following are true:
\begin{enumerate} 
\item $d(\mathfrak{P}|\mathfrak{p}) =0$ if, and only if, $e(\mathfrak{P}|\mathfrak{p})=1$. 
\item If $e(\mathfrak{P}|\mathfrak{p})=3$, which by Theorem \ref{ramification} is equivalent to $v_\mathfrak{p}(a) < 0$ and $(v_{\mathfrak{p}}(a), 3)=1$, then $d(\mathfrak{P}|\mathfrak{p}) =2$. 
\item If $e(\mathfrak{P}|\mathfrak{p})=2$, 
\begin{enumerate} 
\item If $p \neq 2$, by Theorem \ref{ramification}, this occurs precisely when $\Delta$ is not a square in $K$, $a\equiv \pm 2 \mod \mathfrak{p}$,  $( v_{\mathfrak{p}}(\Delta ), 2)=1$. In this case, $2 \;| \;v_{\mathfrak{P}}(\Delta )$ and $d(\mathfrak{P}|\mathfrak{p})=1$.
\item If $p = 2$, by Theorem \ref{ramification}, this occurs when $r\notin K$, $a\equiv 0 \mod \mathfrak{p}$, there is $w_\mathfrak{p} \in K$ such that $v_{\mathfrak{p}}(\left( \frac{1}{a^2}+1 -w_\mathfrak{p}^2+w_\mathfrak{p} \right) , 2)=1\quad\text{and}\quad v_{\mathfrak{p}}\left( \frac{1}{a^2}+1 -w_\mathfrak{p}^2+w_\mathfrak{p} \right)<0$. Also, in this case, there exists $\eta_\mathfrak{P} \in L$ such that $v_{\mathfrak{P}}\left( \frac{1}{a^2}+1 -\eta_\mathfrak{P} ^2+\eta_\mathfrak{P} \right)\geq 0,$ and we have for this $\mathfrak{P}$ that $$d(\mathfrak{P}|\mathfrak{p})=-v_{\mathfrak{p}}\left( \frac{1}{a^2} +1 -w_\mathfrak{p}^2+w_\mathfrak{p} \right)+1 .$$
\end{enumerate} 
\end{enumerate} 
\end{lemma} 
\begin{proof}
Let $\mathfrak{p}$ be a place of $K$, $\mathfrak{P}_r$ a place of $L(r)$ above $\mathfrak{p}$,  $\mathfrak{P}=\mathfrak{P}_r\cap L$, and $\mathfrak{p}_r=\mathfrak{P}_r\cap K(r)$.
 \begin{enumerate} 
\item  As the constant field $\mathbb{F}_q$ of $K$ is perfect, all residue field extensions in $L/F$ are automatically separable. The result then follows from \cite[Theorem 5.6.3]{Vil}.
\item If $e(\mathfrak{P}|\mathfrak{p})=3$, then as $p \nmid 3$, it follows again from [Theorem 5.6.3, Ibid.] that $d(\mathfrak{P}|\mathfrak{p}) = e(\mathfrak{P}|\mathfrak{p}) - 1 =   2$.
\item When $e(\mathfrak{P}|\mathfrak{p})=2$, 
\begin{enumerate} 
\item  if $p\neq 2$, then by [Theorem 5.6.3, Ibid.], $d(\mathfrak{P}|\mathfrak{p}) = e(\mathfrak{P}|\mathfrak{p}) - 1 =   1$.
\item if $p= 2$, then we work on the tower $L(r)/K(r)/ K$. If $e(\mathfrak{P}|\mathfrak{p})=2$, then $e(\mathfrak{p}_r|\mathfrak{p})=2$, $e(\mathfrak{P}_r|\mathfrak{p}_r)=1$ and $e(\mathfrak{P}_r|\mathfrak{P})=1$. As $p=2$, the extension $K(r)/K$ is Artin-Schreier and is generated by an element $\alpha$ such that $\alpha^2 - \alpha = \frac{1}{a^2}+1$. By Artin-Schreier theory (see \cite[Theorem 3.7.8]{Sti}), as $ e(\mathfrak{p}_r|\mathfrak{p})=2$, there exists an element $w_\mathfrak{p} \in K$ such that $$(v_{\mathfrak{p}}\left( \frac{1}{a^2}+1 -w_\mathfrak{p}^2+w_\mathfrak{p} \right) , 2)=1\qquad\text{and}\qquad v_{\mathfrak{p}}\left( \frac{1}{a^2}+1 -w_\mathfrak{p}^2+w_\mathfrak{p} \right)<0.$$ 
In addition, since $e(\mathfrak{P}_r|\mathfrak{P})=1$, there exists $\eta_\mathfrak{P} \in L$ such that $$v_{\mathfrak{P}}\left( \frac{1}{a^2}+1 -\eta_\mathfrak{P} ^2+\eta_\mathfrak{P} \right)\geq 0.$$ By Artin-Schreier theory (see \cite[Theorem 3.7.8]{Sti}), we obtain
 $$d(\mathfrak{p}_r | \mathfrak{p})=  -v_{\mathfrak{p}}\left( \frac{1}{a^2}+1 -w_\mathfrak{p}^2+w_\mathfrak{p} \right)+1.$$ By \cite[Theorem 5.7.15]{Vil}, we then find by equating differential exponents in the towers $L(r)/K(r)/K$ and $L(r)/L/K$ that
 $$ d( \mathfrak{P}_r| \mathfrak{p})= d( \mathfrak{P}_r| \mathfrak{P})+ e( \mathfrak{P}_r| \mathfrak{P}) d( \mathfrak{P}| \mathfrak{p})= d( \mathfrak{P}_r| \mathfrak{p}_r)+ e( \mathfrak{P}_r| \mathfrak{p}_r) d( \mathfrak{p}_r| \mathfrak{p}).$$
This implies that
 $$  d( \mathfrak{P}| \mathfrak{p})=  d( \mathfrak{p}_r| \mathfrak{p})= -v_{\mathfrak{p}}\left( \frac{1}{a^2}+1 -w_\mathfrak{p}^2+w_\mathfrak{p} \right)+1,$$
as $e(\mathfrak{P}_r|\mathfrak{P})=e(\mathfrak{P}_r|\mathfrak{p}_r)=1$ implies $d( \mathfrak{P}_r| \mathfrak{P})=d(\mathfrak{P}_r|\mathfrak{p}_r)=0$.
\end{enumerate} 
\end{enumerate} 
\end{proof}
We are now able to state and prove the Riemann-Hurwitz formula for this cubic form. 
\begin{theoreme}[Riemann-Hurwitz II] \label{RH} Let $p \neq 3$. Let $L/K$ be a cubic  geometric extension and $y$ a primitive element of $L$ with minimal polynomial $f(X) = X^3-3X-a$. Let $\Delta=-27(a^2-4)$ be the discriminant of $f(X)$ and $r$ a root of the quadratic resolvent $R(X)= X^2 +3aX + ( -27 + 9 a^2)$ of the cubic polynomial $X^3-3X-a$ in $\overline{K}$.
Then the genus $g_L$ of $L$ is given according to the formula
\begin{enumerate} 
\item If $p \neq 2$, then $$g_L =  3g_K - 2 +\frac{1}{2}\sum_{\mathfrak{p}\in \mathcal{S}} d_{K}(\mathfrak{p})+  \sum_{\substack{ v_\mathfrak{p}(a)<0\\ ( v_\mathfrak{p}(a),3) = 1}} d_{K}(\mathfrak{p}).$$
where $\mathcal{S}$ is the set of places of $K$ such that both $a\equiv \pm 2 \mod \mathfrak{p}$ and $  v_{\mathfrak{p}}(\Delta, 2)=1$. Moreover, $\Delta$ is a square in $K$ up to a unit if, and only if, the set $\mathcal{S}$ is empty.
\item If $p = 2$, then $$g_L =  3g_K - 2 +  \frac{1}{2} \sum_{\mathfrak{p}\in \mathcal{S}} [ -v_{\mathfrak{p}}\left( \frac{1}{a}+1 -w_\mathfrak{p}^2+w_\mathfrak{p} \right)+1] d_{K}(\mathfrak{p})+  \sum_{\substack{ v_\mathfrak{p}(a)<0\\  ( v_\mathfrak{p}(a),3) = 1}} d_{K}(\mathfrak{p}),$$
where $\mathcal{S}$ is the set of places of $K$ such that both $a \equiv 0 \mod \mathfrak{p}$ and there exists $w_\mathfrak{p} \in K$ such that $v_{\mathfrak{p}}\left( \frac{1}{a}+1 -w_\mathfrak{p}^2+w_\mathfrak{p} \right)<0$ and $(v_{\mathfrak{p}}\left( \frac{1}{a}+1 -w_\mathfrak{p}^2+w_\mathfrak{p} \right) , 2)=1$. Moreover, if $r \in K$ (hence the extension $L/K$ is Galois), then the set $\mathcal{S}$ is empty.
\end{enumerate}
\end{theoreme}
\begin{proof} \begin{enumerate} \item By \cite[Theorem 9.4.2]{Vil}, the term associated with a place $\mathfrak{P}$ of $L$ in the different $\mathfrak{D}_{L/F}$ contributes $\frac{1}{2} d_L(\mathfrak{P})^{d(\mathfrak{P}|\mathfrak{p})}$ to the genus of $L$, where $\mathfrak{p}$ is the place of $K$ below $\mathfrak{P}$, $d_L(\mathfrak{P})$ is the degree of the place $\mathfrak{P}$, and $d(\mathfrak{P}|\mathfrak{p})$ is the differential exponent of $\mathfrak{P}|\mathfrak{p}$. By the fundamental identity $\sum_i e_i f_i = [L:K] = 3$ for ramification indices $e_i$ and inertia degrees $f_i$ of all places of $L$ above $\mathfrak{p}$, we always have that $f_i=1$ whenever $\mathfrak{p}$ ramifies in $L$ (fully or partially). Thus from Lemma  \ref{diffexpqneq1mod3}, it follows that $d(\mathfrak{P}|\mathfrak{p}) = 2$ if $\mathfrak{p}$ is fully ramified, whereas $d(\mathfrak{P}|\mathfrak{p}) = 1$ if $\mathfrak{p}$ is partially ramified. The result then follows by reading off [Theorem 9.4.2, Ibid.] and using the conditions of Lemma \ref{diffexpqneq1mod3}. \item This follows in a manner similar to part (1) of this theorem, via Lemma \ref{diffexpqneq1mod3} for $p=2$. \end{enumerate}
\end{proof}
We obtain directly the following corollary when the extension $L/K$ is Galois.
\begin{corollaire} \label{RHGalois} Let $p \neq 3$. Let $L/K$ be a Galois cubic  geometric extension and $y$ a primitive element of $L$ with minimal polynomial $f(X) = X^3-3X-a$. Then the genus $g_L$ of $L$ is given according to the formula $$g_L =  3g_K - 2 +  \sum_{\substack{v_\mathfrak{p}(a)<0\\  ( v_\mathfrak{p}(a),3) = 1}} d_{K}(\mathfrak{p}).$$
\end{corollaire}
\subsection{$X^3 +aX +a^2$, $a\in K$, $p=3$} 

\begin{lemma}\label{char3diffexp}
Suppose that $p = 3$. Let $L/K$ be a separable cubic extension and $y$ a primitive element with minimal polynomial $X^3 + aX +  a^2$. Let $\mathfrak{p}$ be a place of $K$ and $\mathfrak{P}$ a place of $L$ above $\mathfrak{p}$. 
\begin{enumerate}
\item  $d(\mathfrak{P}|\mathfrak{p}) =0$ if, and only if, $e(\mathfrak{P}|\mathfrak{p})=1$. 

\item when $e(\mathfrak{P}|\mathfrak{p})=3$, by Theorem \ref{char3localstandardform}, there is $w_\mathfrak{p} \in K$ such that $v_\mathfrak{p}(\alpha_\mathfrak{p}  ) < 0$ and $(v_\mathfrak{p}(\alpha_\mathfrak{p}  ),3)=1$ with $$\alpha_\mathfrak{p}  = \frac{(ja^2 + (w_\mathfrak{p} ^3 + a w_\mathfrak{p} ) )^2}{a^3}.$$ Then $d ( \mathfrak{P}| \mathfrak{p}) = -v_{\mathfrak{p}}(\alpha_\mathfrak{p} )+2$.
\item $d(\mathfrak{P}|\mathfrak{p}) = 1$ whenever $e(\mathfrak{P}|\mathfrak{p})=2$. Moreover, by Lemma \ref{char3localstandardform}, when $e(\mathfrak{P}|\mathfrak{p})=2$, there is generator $z_\mathfrak{p} $ such that $z_\mathfrak{p} ^3 +c_\mathfrak{p}  z_\mathfrak{p}  +c _\mathfrak{p} ^2 =0$ and $v_\mathfrak{p}(c_\mathfrak{p}  ) \geq 0$ and $(v_\mathfrak{p} (c_\mathfrak{p}),2)=1$. 
\end{enumerate}
\end{lemma} 
\begin{proof}
Let $b \in \overline{K}$ such that $b^2 = -a$, $\mathfrak{p}$ be a place of $K$, $\mathfrak{P}_b$ be a place of $L(b)$ above $\mathfrak{p}$, $\mathfrak{p}_b=\mathfrak{P}_b\cap K(b)$, $\mathfrak{P}= \mathfrak{P}_b\cap L$.
\begin{enumerate}
\item This is an immediate consequence of \cite[Theorem 5.6.3]{Vil}.
\item Suppose that $\mathfrak{p}$ is ramified in $L$, whence $\mathfrak{p}_b$ is ramified in $L(b)$. 
Moreover, by Theorem \ref{char3localstandardform}, there exists $w_\mathfrak{p} \in K$ such that $v_\mathfrak{p} (\alpha_\mathfrak{p}) <0$ and $(v_\mathfrak{p} (\alpha_\mathfrak{p}),3)=1$, 
where 
$$\alpha_\mathfrak{p} =  \frac{(ja^2 + (w_\mathfrak{p}^3 + a w_\mathfrak{p}) )^2}{a^3},$$
and furthermore, there exists a generator $z_\mathfrak{p}$ of $L$ such that $z_\mathfrak{p}^3 + \alpha_\mathfrak{p} z_\mathfrak{p} + \alpha_\mathfrak{p}^2=0$. Again by [Theorem 5.6.3, Ibid.], the differential exponent $d(\mathfrak{p}_b | \mathfrak{p})=d(\mathfrak{P}_b | \mathfrak{P})$ of $\mathfrak{p}$ over $K(b)$ (resp. $\mathfrak{P}$ over $L(b)$) is equal to
\begin{enumerate} 
\item $1$ if $\mathfrak{p}$ is ramified in $K(b)$, whence $e(\mathfrak{p}_b | \mathfrak{p})=e(\mathfrak{P}_b | \mathfrak{P})=2$, and 
\item $0$ if $\mathfrak{p}$ is unramified in $K(b)$, whence $e(\mathfrak{p}_b | \mathfrak{p})=e(\mathfrak{P}_b | \mathfrak{P})=1$.
\end{enumerate}
 By \cite[Theorem 2.3]{Con}, $L(b)/K(b)$ is Galois and $-\alpha_\mathfrak{p}$ is a square in $K(b)$. We write $-\alpha_\mathfrak{p} = \beta_\mathfrak{p}^2$. Moreover, $w_\mathfrak{p}= \frac{z_\mathfrak{p}}{\beta_\mathfrak{p}}$ and $w_\mathfrak{p}^3 - w_\mathfrak{p}- \beta_\mathfrak{p}=0$. Moreover, 
 $$v_{\mathfrak{p}_b}(\beta_\mathfrak{p})= \frac{ v_{\mathfrak{p}_b}(\alpha_\mathfrak{p})}{2}=\frac{ e(\mathfrak{p}_b | \mathfrak{p})v_{\mathfrak{p}}(\alpha_\mathfrak{p})}{2} $$
 with $e(\mathfrak{p}_b | \mathfrak{p})=2$ or $1$, depending on whether $\mathfrak{p}$ is ramified or not in $K(b)$. Also, $v_{\mathfrak{p}_b}(\beta_\mathfrak{p})=v_{\mathfrak{p}}(\alpha_\mathfrak{p})$ when $\mathfrak{p}$ is ramified in $K(b)$, whereas $v_{\mathfrak{p}_b}(\beta_\mathfrak{p})= \frac{ v_{\mathfrak{p}}(\alpha_\mathfrak{p})}{2} $ when $\mathfrak{p}$ is unramified in $K(b)$ (note that in this case $2|v_{\mathfrak{p}}(\alpha_\mathfrak{p})$).  Thus $v_{\mathfrak{p}_b}(\beta_\mathfrak{p})<0$ and $(v_{\mathfrak{p}_b}(\beta_\mathfrak{p}),3)=1$ and by \cite[Theorem 3.7.8]{Sti}, we also have that the differential exponent $d(\mathfrak{P}_b | \mathfrak{p}_b)$ of $\mathfrak{p}_b$ in $L(b)$ satisfies 
 $$d(\mathfrak{P}_b | \mathfrak{p}_b)=2 (-v_{\mathfrak{p}}(\beta_\mathfrak{p})+1).$$ 
By \cite[Theorem 5.7.15]{Vil}, the differential exponent of $\mathfrak{p}$ in $L(b)$ satisfies
 $$d(\mathfrak{P}_b | \mathfrak{p})= d(\mathfrak{P}_b | \mathfrak{p}_b) + e(\mathfrak{P}_b | \mathfrak{p}_b) d(\mathfrak{p}_b | \mathfrak{p})= d(\mathfrak{P}_b | \mathfrak{P}) + e(\mathfrak{P}_b | \mathfrak{P}) d(\mathfrak{P} | \mathfrak{p}).$$
 Thus,
 \begin{enumerate} 
 \item if $\mathfrak{p}$ is ramified in $K(b)=K(\beta_\mathfrak{p})$, that is, $(v_\mathfrak{p}(\alpha_\mathfrak{p}),2)=1$ by \cite[Proposition 3.7.3]{Sti}, then $ 2 (-v_{\mathfrak{p}}(\alpha_\mathfrak{p})+1) + 3= 1 +  2d(\mathfrak{P} | \mathfrak{p})$ and 
 $$d(\mathfrak{P} | \mathfrak{p})= -v_{\mathfrak{p}}(\alpha_\mathfrak{p})+2,$$ 
 whereas
  \item if $\mathfrak{p}$ is unramified in $K(b)$, that is, $2|v_\mathfrak{p}(\alpha_\mathfrak{p})$ again by \cite[Proposition 3.7.3]{Sti}, then also
  $$   d(\mathfrak{P} | \mathfrak{p})= 2 \left( -\frac{v_{\mathfrak{p}}(\alpha_\mathfrak{p})}{2}+1\right)=  -v_{\mathfrak{p}}(\alpha_\mathfrak{p})+2.$$
 \end{enumerate}
 \item This is immediate from Theorem  \ref{char3localstandardform} and \cite[Theorem 5.6.3]{Vil}, via application of the same method as in  Lemma \ref{diffexpqneq1mod3} $(3)$.
\end{enumerate} 
\end{proof} 
Finally, we use this to conclude the Riemann-Hurwitz formula for cubic extensions in characteristic 3.
\begin{theorem}[Riemann-Hurwitz III] \label{char3RH}
Suppose that $p = 3$. Let $L/K$ be a separable cubic extension and $y$ a primitive element with minimal polynomial $X^3 + aX + a^2$. Then the genus $g_L$ of $L$ is given according to the formula $$g_L =  3g_K - 2 + \frac{1}{2} \sum_{\mathfrak{p} \in S}\left(-v_{\mathfrak{p}}(\alpha_\mathfrak{p} )+2 \right) d_{K}(\mathfrak{p}) + \frac{1}{2} \sum_{\mathfrak{p} \in T} d_{K}(\mathfrak{p}),$$

where \begin{enumerate} \item $S$ is the set of places of $K$ for which there exists $w_\mathfrak{p}  \in K$ such that $v_\mathfrak{p}(\alpha_\mathfrak{p}  ) < 0$, $(v_\mathfrak{p}(\alpha_\mathfrak{p}  ),3)=1$ with $$\alpha_\mathfrak{p}  = \frac{(ja^2 + (w_\mathfrak{p} ^3 + a w_\mathfrak{p} ) )^2}{a^3},$$ 
and 
\item $T$ is the set of places of $K$ for which there is generator $z_\mathfrak{p} $ such that $z_\mathfrak{p} ^3 +c_\mathfrak{p}  z_\mathfrak{p}  +c_\mathfrak{p}  ^2 =0$,  $v_\mathfrak{p}(c_\mathfrak{p}  ) \geq 0$ and $(v_\mathfrak{p} (c_\mathfrak{p}), 2)=1$. \end{enumerate}
\end{theorem}
\begin{proof} This follows from Lemma \ref{char3diffexp}, \cite[Theorem 9.4.2]{Vil}, and the fundamental identity $\sum e_i f_i = [L:K] = 3$. 
\end{proof}

\section*{Appendix: Algorithm for computing the genus of a cubic equation over $\mathbb{F}_q(x)$} 
In this all section $K = \mathbb{F}_q(x)$ and $L/K$ denotes a cubic extension. 
\subsubsection{Transforming any general cubic polynomial into our forms} 
In this section, we keep all the previous notations. We are given a cubic extension with a generator $y$, whose minimal polynomial is of the form $X^3 + e X^2 + f X + g$. We will first transform this generator into one of our forms. \\ \\ 

\begin{center}
{\bf Algorithm 1: Takes  $x^3 + e x^2 + f x + g$ and $p$} 
\end{center}
\begin{enumerate}
\item {\bf Case $p\neq 3$, }
\begin{enumerate} 
\item  if $3eg= f^2$,\\ {\bf RETURN} $x^3 -a$ with $a =  \frac{27g^3}{-27g^2+f^3}$ and $z=  \frac{3g y}{fy + 3g}$.
\item otherwise, \\ 
{\bf RETURN} $x^3 -3x -a$ with $a=-2-\frac{ ( 27 g^2 -9efg +2f^3)^2}{(3ge-f^2)^3}$ and $z= \frac{-(6efg-f^3-27g^2)y+3g(3eg-f^2)}{(3eg-f^2)(fy+3g)}$.

{\small {\it Note} that if $27 g^2 -9efg +2f^3=0$, the cubic polynomial is reducible, which violates the assumption that it must be irreducible. For the same reason, for instance, $g$ cannot be $0$. Note that this case is also not necessarily disjoint from (a); see also \cite[Theorem 2.1]{MWcubic3}.}  
\end{enumerate} 
\item {\bf Case $p=3$}
\begin{enumerate} 
\item if $e=f=0$, \\
{\bf RETURN} $x^3 -a$, with $a =g$ ; 
\item otherwise, \\
{\bf RETURN} $x^3 + a x + a^2$ with 
\begin{enumerate} 
\item if $e=0$ and $f \neq 0$,\\
 $a =  \frac{g^2}{f^3}$ and $z= \frac{g}{f^2} y$;
\item if $e\neq 0$ and $f=0$,\\
 $a =  \frac{g}{e^3}$ and $z= \frac{g}{e^2x}$;
\item otherwise,\\ 
 $a =  \frac{-f^2e^2+ge^3+f^3}{e^6}$ and $z= \frac{-f^2e^2+ge^3+f^3}{e^4(ey -f)}$. \\
 {\small  {\it Note} that if $-f^2e^2+ge^3+f^3=0$, the cubic polynomial is again reducible, contrary to the assumption that it must be an irreducible polynomial.} 
\end{enumerate}
\end{enumerate} 
\end{enumerate}

\subsubsection{$p\neq 3$} 

\begin{center} 
{\bf {\sc Case 1: The Algorithm 1 returned a form $X^3 -a$. } }
\end{center} 
The next algorithm returns the list of ramified places, indices of ramification, differential exponents, and the value of the genus for purely cubic extensions. \\ \bigskip

\begin{center} 
{\bf Algorithm 2: takes  $x^3-a $ and $p\neq 3$} 
\end{center} 
Use the {\sc factorization Algorithm} to factor $a$ into 
$$\frac{  \prod_{i=1}^s p_i(x)^{e_i} }{ \prod_{i=1}^t q_i(x)^{f_i} },$$ 
with $p_i(x), q_j(x)$ distinct irreducible polynomials and $e_i$ and $f_j$ natural numbers, $i \in \{1, \cdots, s\}$ and $j \in \{1, \cdots , t\}$. 

{\bf RETURN} 
\begin{itemize}
\item[$\cdot$] {\bf List of triples (ramified place, index of ramification, different exponent):}
\begin{itemize} 
\item[$\cdot$]  if $ 3\mid \left( \sum_{i=1}^s e_i \deg (p_i (x))  - \sum_{i=1}^i f_i \deg (q_i (x)) \right)$
$$\{  (p_i(x), 3, 2),  (q_j(x), 3, 2), i \in \{ 1, \cdots , s\} \ with \ 3 \nmid e_i , j \in \{ 1, \cdots , t\} \ with \ 3 \nmid f_i  \} $$ 
\item[$\cdot$] otherwise, 
$$\{  (p_i(x), 3, 2),  (q_j(x), 3, 2), (\infty, 3, 2), i \in \{ 1, \cdots , s\} \ with \ 3 \nmid e_i , j \in \{ 1, \cdots , t\} \ with \ 3 \nmid f_i  \} $$ 
\end{itemize} 
\item[$\cdot$] {\bf Genus of the extension}
\begin{itemize} 
\item[$\cdot$]  if $ 3\mid \left( \sum_{i=1}^s e_i \deg (p_i (x))  - \sum_{i=1}^i f_i \deg (q_i (x)) \right)$, {\small (i.e., the place at infinity is unramified,)} then 
$$g = -2 + \sum_{ \text{ for \  $i$ \ such \ that \ $3 \nmid e_i$}}   \deg( p_i(x) ) + \sum_{ \text{ for \  $i$ \ such \ that \ $3 \nmid f_i$}}   \deg( q_i(x)) $$ 
\item[$\cdot$] otherwise, $$g = -1 + \sum_{ \text{ for \  $i$ \ such \ that \ $3 \nmid e_i$}}   \deg( p_i(x) ) + \sum_{ \text{ for \  $i$ \ such \ that \ $3 \nmid f_i$}}   \deg( q_i(x))   .$$ \\ \bigskip
\end{itemize} 
\end{itemize}

{\small {\it Note} one could  also easily check such an extension is Galois indeed it suffices to check if $q \equiv \ \pm 1 \ mod \ 3$ as it is Galois if and only if $\mathbb{F}_q (x)$ contains a third root of unity. The later being equivalent to $q \equiv \ 1 \  mod \ 3$ (see \cite[Theorem 4.2]{MWcubic3}).  \\ \bigskip}

\begin{center} 
{\sc Application: finding integral basis for a purely cubic extension} 
\end{center}
The statement used for this Algorithm is done in \cite[Theorem 3]{MadMad}, and finds explicitly an integral basis for any purely cubic extension.\\ \bigskip

\begin{center}
{\bf Algorithm finding integral basis takes $y$ (generator of $L/K$ with minimal polynomial), $X^3-a $ and $p\neq 3$}
\end{center}

Use the {\sc factorization Algorithm} to factor $a$ into 
$$\frac{  \prod_{i=1}^s p_i(x)^{e_i} }{ \prod_{i=1}^t q_i(x)^{f_i} },$$ 
with $p_i(x), q_j(x)$ distinct irreducible polynomials and $e_i$ and $f_j$ natural numbers, $i \in \{1, \cdots, s\}$ and $j \in \{1, \cdots , t\}$. 

Use {\sc Euclidean Algorithm} to find $\lambda_i$, $\lambda_i'$ integers and $r_i, \ r_i' \in \{ 0 , 1, 2\}$ such that 
$$e_i = 3 \lambda_i + r_i$$ 
and 
$$f_i = -3 \lambda_i' - r_i'$$

{\bf RETURN} {\bf Integral basis for $L/K$:}
$$\{ \theta_0, \theta_1 , \theta_2\}$$
where
$$\theta_j = \frac{y^j}{\prod_{i=1}^s p_i(x)^{s_{ij}}\prod_{i=1}^t q_i(x)^{s_{ij}'}}$$ 
with $j = 0, 1, 2$, $s_{ij}= \big[ \frac{j r_i}{3}\big] + j \lambda_i$ is the greatest integer not exceeding $ \frac{-j r_i}{3}$ and $s_{ij}'= \big[ \frac{j r_i'}{3}\big] + j \lambda_i'$ is the greatest integer not exceeding $ \frac{-j r_i'}{3}$ .  \\ \bigskip

\begin{center} 
{\bf {\sc Case 2: The Algorithm 1 returned a form $x^3 -3x-a$. }} 
\end{center} 
Note that a cubic extension defined by an irreducible polynomial of the form $x^3-3x-a$ is not necessarily impurely cubic, see \cite[Theorem 2.1]{MWcubic3}. If the reader wishes to first decide if the extension is impurely cubic or not, and if it is not, transform it to a purely cubic extension and go back to Case 1, then s/he could use the following algorithm. Otherwise, the reader can also go directly to Algorithm 3. \\ \bigskip

{\bf Optional algorithm: takes $x^3 -3x-a$ and $p\neq 3$} 

\begin{enumerate} 
\item if $p \neq 2$ and 
\begin{enumerate} 
\item if $\Delta =a^2-4$ is not a square (one can use for instance the {\sc factorization algorithm} to test this),\\  
{\bf RETURN} $L/K$ is impurely cubic.
\item $\Delta= a^2-4$ is a square (one can use for instance the {\sc factorization algorithm} to test this) and determine $\sqrt{\Delta}$ such that $\sqrt{\Delta}^2 = \Delta$,\\  
{\bf RETURN} $L/K$ is purely cubic, taking $c= \frac{ -b + \sqrt{\Delta}}{2a}$ or $c=\frac{ -b - \sqrt{\Delta}}{2a}$, $U= \frac{cY-1}{Y-c}$, $U^3 - c$ irreducible polynomial for $L/K$. 
\end{enumerate} 
\item if $p=2$ and 
\begin{enumerate} 
\item  if $X^2 +aX +1$ has no root in $\mathbb{F}_q(x)$ (testing this requires an algorithm permitting one to check for roots for quadratic polynomials over $\mathbb{F}_q(x)$ in characteristic $2$.) \\ 
{\bf RETURN} $L/K$ is not purely cubic. 
\item if $X^2 + aX+1$ has a root in $\mathbb{F}_q(x)$, compute a root $c$ for this polynomial (this requires an algorithm finding root for quadratic polynomials over $\mathbb{F}_q(x)$ in characteristic $2$.) \\
{\bf RETURN} $L/K$ is purely cubic, $c$ a root of $X^2 + X+1/a$,  $U= \frac{cY-1}{Y-c}$, and $U^3 =c$ irreducible polynomial for $L/K$.  \\ \bigskip
\end{enumerate} 
\end{enumerate}

{\small {\it Note} that one can also to check if an extension with minimal polynomial $x^3 -3x -a$ is a Galois extension, 
\begin{enumerate} 
\item {\bf when $p \neq 2$,} one will only need to check if $-27(a^2-4)$ is a square or not in $\mathbb{F}_q(x)$, which is achievable with the {\sc factorization algorithm}, for instance. Once one knows it is Galois, that is, when $-27(a^2-4)$ is a square, one writes $\delta = \sqrt{-27 (a^2 -4)}$, taking $b = -\frac{1}{2} + \frac{9}{\delta} + \frac{9a}{2\delta}$ or $b=-\frac{1}{2} - \frac{9}{\delta} + \frac{9a}{2\delta}$  and
$$a = \frac{2 b^2 + 2 b -1}{ b^2 +b +1}$$ 
\item {\bf when $p=2$,} one will need to check if $R(x)= x^2 +a x + (1+a^2)$ (quadratic resolvent) has a root or not in $\mathbb{F}_q(x)$ (this would require an algorithm finding roots for quadratic polynomial in $\mathbb{F}_q(x)$ in characteristic $2$). Once one knows it is Galois, that is when this polynomial has a root in $\mathbb{F}_q(x)$, one computes a root for the quadratic resolvent. Call such a root $r$; taking $b= \frac{r}{a}$, one writes 
$$a = \frac{1}{ b^2 +b +1}$$ \\ \bigskip
\end{enumerate}} 

The next algorithm returns the list of ramified places, indices of ramification, differential exponents, and the value of the genus for cubic extensions with minimal polynomial of the form $x^3 -3x -a$. \\ \bigskip

\begin{center} 
{\bf Algorithm 3: takes  $x^3-3x-a $ and $p\neq 3$} 
\end{center} 
Use the {\sc factorization Algorithm} to factor $a$ into  $$\frac{f(x) }{ \prod_{i=1}^t q_i(x)^{f_i} },$$ with $q_i(x)$ distinct irreducible polynomials, $f(x)$ polynomial, $( f(x), \prod_{i=1}^t q_i(x)^{f_i} ) =1$ and $f_i$ natural number, for $i \in \{ 1, \cdots , t\}$. 

\begin{enumerate} 
\item {\bf Case $p\neq 2$}, Use the {\sc factorization Algorithm} to factor $a^2 -4$ into  
$$a^2 -4= \frac{\prod_{i=1}^r r_i(x)^{g_i}}{ \prod_{i=1}^k s_i(x)^{h_i} },$$ 
with $r_i(x), s_j (x)$ distinct irreducible polynomials, $g_i$, $h_i$ natural numbers for $i \in \{ 1, \cdots, r\}$, $j \in \{ 1, \cdots , k\}$. 

{\bf RETURN} 
\begin{itemize}
\item[{\bf Case $\infty$ unramified:}] When $3 | deg (a^2 -4)$,
\begin{enumerate} 
\item[$\cdot$]   {\bf List of triples (ramified places, indices of ramification, differential exponents): }
$$\begin{array}{lll} \{  (q_i(x), 3, 2),  (r_j(x), 2, 1), (s_u(x) , 2, 1), i \in \{ 1, \cdots , t \} \ with \ 3 \nmid e_i  , \\ 
   j \in \{ 1, \cdots , r\} \ with \ 2 \nmid g_j, \ u \in \{ 1, \cdots , k\}  \ with \ 2 \nmid h_u  \}   \end{array} $$ 
  \item[$\cdot$] {\bf Genus of the extension:} 

\begin{align*} g &=  -2 + \frac{1}{2}  \sum_{ \text{ for \  $i$ \ such \ that \ $2 \nmid g_i$}}   \deg( r_i(x)) + \frac{1}{2}  \sum_{ \text{ for \  $i$ \ such \ that \ $2 \nmid h_i$}}   \deg( s_i(x)) \\ 
&\qquad + \sum_{ \text{ for \  $i$ \ such \ that \ $3 \nmid f_i$}}   \deg( q_i(x)) \end{align*}
\end{enumerate} 
\item[{\bf Case $\infty$ ramified:}]Otherwise,
\begin{enumerate}
\item[$\cdot$]  {\bf List of triples (ramified places, indices of ramification, differential exponents): }
  $$\begin{array}{lll} \{  (q_i(x), 3, 2),  (r_j(x), 2, 1), (s_u(x) , 2, 1), (\infty, 2, 1),  i \in \{ 1, \cdots , t \} \ with \ 3 \nmid e_i  , \\ 
   j \in \{ 1, \cdots , r\} \ with \ 2 \nmid g_j, \ u \in \{ 1, \cdots , k\}  \ with \ 2 \nmid h_u  \}\end{array} $$ 
\item[$\cdot$]  {\bf Genus of the extension:} 
\begin{align*} g &=  -3/2 + \frac{1}{2}  \sum_{ \text{ for \  $i$ \ such \ that \ $2 \nmid g_i$}}   \deg( r_i(x)) + \frac{1}{2}  \sum_{ \text{ for \  $i$ \ such \ that \ $2 \nmid h_i$}}   \deg( s_i(x)) \\ 
&\qquad + \sum_{ \text{ for \  $i$ \ such \ that \ $3 \nmid f_i$}}   \deg( q_i(x)) \end{align*}
\end{enumerate} 
\end{itemize} 
\item {\bf Case $p=2$}, Do {\bf Algorithm Artin-Schreier} below giving it the polynomial $x^2 - x -\frac{1}{a}-1$, it will return, in particular,
$$b = \frac{ g(x)}{ \prod_{j=1}^s p_{i_j} (x)^{\alpha_{i_j}^t}} $$ 
with $g(x)$ polynomial, $(g(x) ,  \prod_{j=1}^s p_{i_j} (x)^{\alpha_{i_j}^t})=1$ and $2 \nmid \alpha_{i_j}^t$. 

{\bf RETURN} 
\begin{itemize}
\item[{\bf Case $\infty$ unramified:}] If $3 | \deg (b)$, 
\begin{itemize} 
\item[$\cdot$]  {\bf List of triples (ramified places, indices of ramification, differential exponents): } 
$$\begin{array}{lll} \{  (q_i(x), 3, 2),  (p_{i_j}(x), 2, \alpha^t_{i_j}+1),  i \in \{ 1, \cdots , t \} \ \text{with}  \ 3 \nmid f_i  ,  j \in \{ 1, \cdots , s \} \}
   \end{array} $$ 
\item[$\cdot$] {\bf Genus of the extension:}
\begin{align*} g =  -2 + \frac{1}{2}  \sum_{j=1}^s ( \alpha^t_{i_j}+1) \deg(p_{i_j}(x))    + \sum_{ \text{ for \  $i$ \ such \ that \ $3 \nmid e_i$}}   \deg( q_i(x)) \end{align*}  \\ \bigskip
\end{itemize} 
\item[{\bf Case $\infty$ ramified:}]  if $3 \nmid \deg(b)$,
\begin{itemize} 
\item[$\cdot$]  {\bf List of triples (ramified places, indices of ramification, differential exponents): } 
\begin{align*} \{  &(q_i(x), 3, 2),  (p_{i_j}(x), 2, \alpha^t_{i_j}+1),  (\infty, 2, deg(b)+1),  \\& i \in \{ 1, \cdots , t \} \ \text{with}  \ 3 \nmid e_i  ,  j \in \{ 1, \cdots , s \} \}
   \end{align*} 
\item[$\cdot$] {\bf Genus of the extension;}
\begin{align*} g =  -\frac{3}{2} - \frac{1}{2} \deg(b) + \frac{1}{2}  \sum_{i=1}^s ( \alpha^t_{i_j}+1) \deg(p_{i_j}(x))    + \sum_{ \text{ for \  $i$ \ such \ that \ $3 \nmid e_i$}}   \deg( q_i(x))   \end{align*}  \\ 
\end{itemize}
\end{itemize} 
\end{enumerate} 

The following algorithm is an intermediate algorithm used in the previous algorithm that computes ramified places for a Artin-Schreier extension.  More precisely, given an Artin-Schreier extension of prime degree $r$, that is, a extension with a generator $y$ such that its minimal polynomial is of the form $x^r -x -a$ called Artin-Schreier, this algorithm will find an Artin-Schreier generator $z$ such that its minimal polynomial if of the form $x^r -x -b$ where $b=a + \eta^r - \eta$ for some $\eta \in \mathbb{F}_q(x)$  and 
$$b = \frac{g(x) }{ \prod_{i=1}^k s_i(x)^{h_i} }$$ 
with $g(x)$ polynomial, $s_i(x)$ distinct irreducible polynomials, $h_i$ natural numbers with $r \nmid h_i$, $i \in \{1, \cdots, k\}$. This algorithm is a consequence of \cite[Example 5.8.8]{Vil}; we add it for completeness.
\\ \bigskip

\begin{center} 
{\bf Algorithm Artin-Schreier: takes $y$ (a generator for the Artin-Schreier extension such that its minimal polynomial is), $x^r-x-a $ and $r$ prime number } 
\end{center}

\begin{enumerate} 
\item[{\bf Step 1}]

Use the {\sc factorization Algorithm} to factor $a$ into 
$$\frac{ f(x)}{ \prod_{i=1}^t p_i(x)^{\alpha_i} },$$ with $p_i(x)$ distinct irreducible polynomials, $f(x)$ polynomial and $( f(x),  \prod_{i=1}^t p_i(x)^{\alpha_i})=1$, $\alpha_i$ natural numbers, $i \in \{ 1, \cdots t\}$. \\


For each $i \in \{ 1, \cdots , t \}$, we denote $\mathfrak{p}_i$ to be the finite place associated to $p_i(x)$. \\ 

Do {\sc {\bf Step 2}}, for each $i$ such that $ r | \alpha_i$. 

\item[{\bf Step 2}] Given $i$ as above,
\item[{\bf Step 2a}] 

Write $\alpha_i = r \lambda_i$, for some $\lambda_i$ natural number.

Do the {\sc algorithm computing the partial fraction decomposition} of $a$ as 
$$ a = s(x)+ \sum_{i=1}^t \sum_{k=0}^{\alpha_i -1} \frac{ t_k^{(i)} (x) }{ p_i (x)^{\alpha_i -k}}$$ 
such that 
$$ \deg ( t_k^{(i)} (x)) < \deg ( p_i (x)), \ for \ k = 0 , 1, \cdots , \alpha_i -1$$ 
 
Write 

$$a = \frac{ t_0^{(i)} (x) }{ p_i (x)^{ 2 \lambda_i }} + t_1 (x)$$ 

{\small {\it Note $$v_{\mathfrak{p}_i} ( t_1 (x) ) > - 2\lambda_i $$ }}

\item[{\bf Step 2b}] Find $m(x) \in k[x]$ such that 
$$ m(x)^r \equiv t^{(i)}_0 (x)  \mod p_i(x)$$

{\small {\it Note} this is possible since $k$ is a perfect field and $[ k[x]/ (p_i (x)) : k ]< \infty$, so that $M= k[x]/ (p_i (x))  $ is perfect, that is, $M^2= M$. The previous step can be achieved by finding a root $\delta$ for $p_i(x)$ in $\mathbb{F}_{q^{\deg(p_i(x))}}$ then computing $ t^{(i)}_0 (\delta )$ in $\mathbb{F}_{q^{\deg(p_i(x))}}$ and finding $\beta \in \mathbb{F}_{q^{\deg(p_i(x))}}$ such that $\beta^r = t^{(i)}_0 (\delta)$ in $\mathbb{F}_{q^{\deg(p_i(x))}}$. Then $m(x) = \beta + p_i(x)$ will satisfy the congruence above. } 

\item[{\bf Step 2c}] We set 
\begin{itemize} 
\item[$\cdot$] $$z = y - \left(\frac{m(x)}{p_i(x)^{\lambda_i}} \right)$$
\item[$\cdot$] $$b =a-  \left(\frac{m(x)}{p_i(x)^{\lambda_i}} \right)^r -\left( \frac{m(x)}{p_i(x)^{\lambda_i}} \right).$$
\end{itemize}

We write $b = p_i(x)^{\alpha'_i}  \frac{ g(x)}{ q(x)}$ with $q(x)$ polynomial and $p_i(x)$ irreducible polynomial and $(p_i(x) ,  q(x))=1$, $(g(x),q(x))=1$, $(p_i(x), g(x))=1$ and $\alpha'_i $ integer. 
\begin{enumerate} 
\item If $\alpha'_i \geq 0$; Change $i$ in {\bf STEP 2} with $a=b$ and $y=z$, if there are no more $i$ to work with, taking $a=b$ and $y=z$ exit to {\bf STEP 3}. 
\item If $\alpha'_i < 0$ with $r \nmid \alpha'_i $; Change $i$ in {\bf STEP 2} with $a=b$ and $y=z$, if there are no more $i$ to work with, taking $a=b$ and $y=z$ exit to {\bf STEP 3}.
\item Otherwise, repeat {\bf Step 2} with taking $a=b$ and $y=z$.
\end{enumerate}


{\small {\it Note} that 
$$v_{\mathfrak{p}_i} ( a + w^r -w) \geq min \{  v_{\mathfrak{p}_i} ( a + w^r ), v_{\mathfrak{p}_i} ( w)\}$$ 

but $v_{\mathfrak{p}_i} ( a + w^r ) > - r\lambda_i$ 
and 
$ v_{\mathfrak{p}_i} (w) = -\lambda_i > -r \lambda_i$.

Therefore,
$$ v_{\mathfrak{p}_i} (a + w^r -w) > v_{\mathfrak{p}_i} ( a ).$$  
{\it Note} that the other prime valuations are affected in a way that does not cause any nonnegative valuations to become negative.
Hence, the process will end in finite time. }
\item[{\bf STEP 3}] We set
 $$s = \deg(a)$$ 
 \begin{enumerate}
 \item If $s \leq 0$, move to {\bf STEP 4}. 
 \item If $s>0$, $r \nmid s$, move to {\bf STEP 4}. 
 \item Otherwise 
 \begin{enumerate} 
\item Use the {\sc euclidean algorithm} to write $s = rd$ with $d$ a natural number,
\item Write 
$$ a = \frac{g(x)}{ h(x)}$$ 
where $g(x)$, $h(x)$ polynomials with $(f(x), g(x))=1$. 
 \item Apply the {\sc Euclidean algorithm} to find $q(x)$ and $r(x)$ polynomial with $ \deg( r(x)) < h(x)$ or $r(x) =0$, and 
 $$ g(x) = h(x) q(x) + r(x)$$ 
 \item Write $q(x) = \alpha x^{rt} + t(x)$ 
 
 {\it {\small Note $\deg(t(x))< rt$.}} 
 \item Find $\beta$ in $\mathbb{F}_q$ such that $\beta ^r= \alpha $.
 \item Take
 \begin{itemize} 
 \item[$\cdot$] $$z  = y-  \beta x^t$$
\item[$\cdot$] $$b =  a- (\beta x^t)^r + \beta x^t $$ 
 \end{itemize} 
 
\item  \begin{itemize} 
 \item[$\cdot$] If $r | \deg(a)$, then redo {\bf STEP 3} with $y=z$ and $a =b$. 
 \item[$\cdot$] Otherwise exit to {\bf STEP 4}. 
 \end{itemize} 
{\small {\it Note $\deg (b) \leq rt- 1<rt$ }.  Note that the other prime valuations are affected in a way that does not cause any nonnegative valuations to become negative. Hence, the entire process will end in finite time.}
 \end{enumerate} 
 \end{enumerate} 
\item[{\bf STEP 4}] {\bf RETURN} 
\begin{itemize} 
\item[$\cdot$] $z$ in terms of initial $y$ given into the algorithm. 
\item[$\cdot$]  $$b = \frac{ g(x)}{ \prod_{j=1}^s p_{i_j} (x)^{\alpha_{i_j}^t}} $$ 
with $g(x)$ polynomial, $\alpha_{i_j}^t$ natural number with $3 \nmid \alpha_{i_j}^t$, for $j \in \{ 1, \cdots , s\}$.\\ 
\end{itemize} 
\bigskip

\end{enumerate} 

\begin{center} 
{\sc Application: finding an integral basis for a cubic extension defined by an irreducible polynomial of the form $X^3 -3X -a$.} 
\end{center}
The statement used for this Algorithm is done in \cite[Theorem 2.1]{MWcubic4} and finds explicitly an integral basis for cubic extensions defined by a irreducible polynomial of the form $X^3 -3X -a$.\\ \bigskip

\begin{center}
{\bf Algorithm integral basis 2: takes $y$ (generator of $L/K$ with minimal polynomial), $X^3-3X-a $ and $p\neq 3$} 
\end{center}
Use the {\sc factorization algorithm} to factor $a$ into 
$$a=\frac{ \alpha}{\gamma^3 \beta},$$ 
where $(\alpha,\beta\gamma) = 1$, $\beta$ polynomial is cube-free, and $\beta = \beta_1 \beta_2^2$, where $\beta_1$ and $\beta_2$ are polynomials square-free. 

\begin{itemize} 
\item[{$\cdot$}] {\bf Case $p\neq 2$}, 
\begin{enumerate} 
\item Use the {\sc factorization algorithm} to factor $4\gamma^6 \beta^2 - \alpha^2$ into 
$$(4\gamma^6 \beta^2 - \alpha^2)=\eta_1 \eta_2^2,$$ where $\eta_1$ is polynomial square-free. 

\item Use {\sc Chinese remainder theorem} to find $T$ and $S$ polynomials such that 
$$
T \equiv -\frac{  \alpha  }{2 \gamma^2 \beta_2}  \mod \eta_2^2\quad\text{ and } \quad T\equiv 0 \mod \beta_1^2;
$$

\item Set $I= \beta_1^2 \eta_1^2$. 
\end{enumerate} 
\item[$\cdot$] {\bf Case $p= 2$}, \\
Use {\sc Factorization Algorithm} to factor $\alpha$ into 
 $$\alpha = \prod_{i=1}^s \alpha_i^{s_i },$$ 
 where $\alpha_i$ distinct irreducible polynomials. \\ \bigskip

 \begin{itemize}
 \item[{\bf Galois case}] If the polynomial $R(X)= X^2 +a X +(1+a^2)$ (quadratic resolvent) has a root in $\mathbb{F}_q(x)$:  \\ 
 Then, 
 \begin{enumerate}
\item Find a root $r$ of $R(X)$. 
 \item Set $I = \beta_1 \alpha$. 
 \item Use {\sc Chinese remainder theorem} to find a polynomial $T$ such that 
$$T \equiv 0 \mod \beta_1 \quad\text{ and }\quad  T \equiv \frac{r}{  \gamma^2 \beta_1^2 \beta_2^2} \mod \alpha;$$
 \end{enumerate} 
\item[{\bf Non Galois Case}] If the polynomial $R(X)= X^2 +a X +(1+a^2)$ (quadratic resolvent) has no root  in $\mathbb{F}_q(x)$.
\begin{enumerate} 
\item Use {\sc Artin-Schreier Algorithm} above with $x^2 -x- \frac{ \gamma^6 \beta_1^2 \beta_2^4 }{ \alpha ^2} $, it will return, in particular
$$b = \frac{ g(x)}{ \prod_{j=1}^s \alpha_{i_j}^{t_{i_j}^t}} $$ 
with $g(x)$ polynomial, $t_{i_j}^t$ natural number such that $3\nmid t_{i_j}^t$. 
\item Set $$I =\beta_1 \alpha \prod_{j=1}^s \alpha_{i_j}^{-\frac{1}{2} (t_{i_j}^t+1) } $$
 \item Use {\sc Chinese remainder theorem} to find a polynomial $T$ such that 
  $$T \equiv 0 \mod \beta_1 \quad \text{ and }\quad T \equiv \frac{\alpha b}{\gamma^2 \beta_2}\mod \alpha \prod_{j=1}^s \alpha_{i_j}^{-\frac{1}{2} (t_{i_j}^t+1) }$$ 
  \end{enumerate} 
\end{itemize} 
\end{itemize} 

{\bf RETURN}  Integral basis of $L/K$
$$\mathfrak{B} = \left\{ 1, \omega+ S, \frac{1}{I} (\omega^2+ T\omega +V  ) \right\}$$  
with $S, V \in \mathbb{F}_q[x]$ and $ V \equiv T^2 -3(\gamma \beta_1 \beta_2)^2 \mod I$.

\subsubsection{$p=3$} 
\begin{center} 
{\sc The Algorithm 1 returned a form $X^3 + a X + a^2$. } \\ \bigskip
\end{center}

The next algorithm returns the list of ramified places, indices of ramification, differential exponents, and the value of the genus for cubic extensions with minimal polynomial of the form $x^3 +a x +a^2$. \\ \bigskip
\begin{center}
{\bf Algorithm 4: takes  $X^3 + a X + a^2$ and $p=3$} \\ \bigskip
\end{center} 
\begin{enumerate} 
\item Use the {\sc factorization algorithm}, to factorize $a$ as 
$$ \frac{ f(x)}{\prod_{i=1}^m p_i(x)^{\alpha_i}}$$ 
with $p_i(x)$ distinct irreducible polynomials, $f(x)$  polynomial and $(f(x) ,   \prod_{i=1}^m p_i(x)^{\alpha_i}) =1$ and $$f(x) = \prod_{i=1}^u r_i(x)^{\beta_i}$$
where $r_i(x)$ distinct irreducible polynomials and $\beta_i$ natural integer, for $i \in \{ 1, \cdots , s\}$. 

\item Use the {\sc Generalized Artin-Schreier Algorithm} below for $X^3 + a X + a^2$, it will return in particular,
 $$b = \frac{ g(x)}{ \prod_{j=1}^s p_{i_j} (x)^{\alpha_{i_j}^t}} $$ 
with $g(x)$ polynomial, $\alpha_{i_j}^t$ natural number with $3 \nmid \alpha_{i_j}^t$, for $j \in \{ 1, \cdots , s\}$.
\end{enumerate} 

{\bf RETURN} 
\begin{itemize}
\item[{\bf Case $\infty$ ramified:}] if $2 \nmid \deg (a)$ (initial $a$ given into algorithm) and $\deg(b) \leq 0$, 
\begin{itemize} 
\item[$\cdot$]  {\bf List of triples (ramified places, indices of ramification, differential exponents): } 
$$\begin{array}{lll} \{   (p_{i_j}(x), 3, \alpha^t_{i_j}+1), (r_i(x), 2 , 1), (p_k(x), 2,1) , (\infty, 2, 1),  j \in \{ 1, \cdots , s \},\\
 i \in \{ 1, \cdots , u \}  \text{ such that } 2 \nmid \beta_i , k \in \{1, \cdots , m \}\backslash \{ i_1, \cdots , i_s\} \}
   \end{array} $$ 
\item[$\cdot$] {\bf Genus of the extension:}
\begin{align*} g =&  -\frac{3}{2} + \frac{1}{2}  \sum_{i=1}^s ( \alpha^t_{i_j}+1) \deg(p_{i_j}(x))    +\frac{1}{2} \sum_{ \text{ for \  $i$ \ such \ that \ $2 \nmid \beta_i$}}   \deg( r_i(x))\\
& +\frac{1}{2} \sum_{ k \in \{1, \cdot , m \}\backslash \{ i_1, \cdots , i_s\}}   \deg( p_k(x))
 \end{align*}  \\ \bigskip
\end{itemize} 
\item[{\bf Case $\infty$ ramified:}]  Otherwise,
\begin{itemize} 
\item[$\cdot$]  {\bf List of triples (ramified places, indices of ramification, differential exponents): } 
$$\begin{array}{lll} \{   (p_{i_j}(x), 3, \alpha^t_{i_j}+1), (r_i(x), 2 , 1), (p_k(x), 2,1) , (\infty, 2, 1),  j \in \{ 1, \cdots , s \},\\
 i \in \{ 1, \cdots , u \}  \text{ such that } 2 \nmid \beta_i , k \in \{1, \cdots , m \}\backslash \{ i_1, \cdots , i_s\} \}
   \end{array} $$ 
\item[$\cdot$] {\bf Genus of the extension:}
\begin{align*} g =&  -2 + \frac{1}{2}  \sum_{i=1}^s ( \alpha^t_{i_j}+1) \deg(p_{i_j}(x))    +\frac{1}{2} \sum_{ \text{ for \  $i$ \ such \ that \ $2 \nmid \beta_i$}}   \deg( r_i(x))\\
& +\frac{1}{2} \sum_{ k \in \{1, \cdots , m \}\backslash \{ i_1, \cdots , i_s\}}   \deg( p_k(x))
 \end{align*}  \\ \bigskip
\end{itemize} 
\end{itemize}

The following algorithm is an intermediate algorithm used in the previous algorithm that computes ramified places for a extension with minimal polynomial $X^3+aX+a^2$.  This algorithm will find a generator $z$ such that its minimal polynomial is of the form $x^3+ b x +b^2$ where $b=  \frac{( a^2 + \eta ^3 + a\eta )^2}{ a^3}$
 for some $\eta \in \mathbb{F}_q(x)$  and 
$$b = \frac{g(x) }{ \prod_{i=1}^k s_i(x)^{h_i} }$$ 
with $g(x)$ polynomial, $ s_i(x)$ distinct irreducible polynomials, $h_i$ natural numbers with $r \nmid h_i$, $i \in \{1, \cdots, k\}$. This algorithm uses the same arguments as the ones used in Theorem  \ref{char3localstandardform} and \cite[Lemma 1.2]{MWcubic3}, one difference is that we also address the place at infinity, which was not done in \cite[Lemma 1.2]{MWcubic3}.
\\ \bigskip

\begin{center}
{\bf Generalized Artin-Schreier Algorithm: takes $y$ (generator for the extension $L/K$ with minimal polynomial),  $X^3 + a X + a^2$ and $p=3$} \\ \bigskip
\end{center} 
\begin{enumerate} 
\item[{\bf Step 1}] Use the {\sc factorization algorithm}, to factorize $a$ as 
$$ \frac{ f(x)}{\prod_{i=1}^t p_i(x)^{\alpha_i}}$$ 
with $p_i(x)$ distinct irreducible polynomials, $f(x)$  polynomial and $(f(x) ,   \prod_{i=1}^t p_i(x)^{\alpha_i}) =1$, $\alpha_i$ natural numbers, $i \in \{ 1, \cdots, t\}$. 


For each $i \in \{ 1, \cdots , t\}$, we denote $\mathfrak{p}_i$ to be the finite place associated to $p_i(x)$.

For each $i$ such that $ 3 | \alpha_i$, do {\bf STEP 2}:

\item[{\bf Step 2}] For some $i$ as above

\item[{\bf Step 2a}]  

Write $\alpha_i = 3 \lambda_i$, for some $\lambda_i$ natural number.

Use the {\sc partial fraction decomposition algorithm} and write 
$$ a = s(x)+ \sum_{i=1}^t \sum_{k=0}^{\alpha_i -1} \frac{ t_k^{(i)} (x) }{ p_i (x)^{\alpha_i -k}}$$ 
such that 
$$ \deg ( t_k^{(i)} (x)) < \deg ( p_i (x)), \ for \ k = 0 , 1, \cdots , \alpha_i -1$$ 

Then, write 

$$a = \frac{ t_0^{(i)} (x) }{ p_i (x)^{ 3 \lambda_i }} + t_1 (x)$$ 
{\small {\it Note
$$v_{\mathfrak{p}_i} ( t_1 (x) ) > -3\lambda_i $$ }}

\item[{\bf Step 2b}]  Find $m(x) \in k[x]$ such that 
$$ m(x)^3 \equiv t^{(i)}_0 (x)^2  \mod p_i(x)$$ 

{\small {\it Note }This is possible since $k$ is a perfect field and $[ k[x]/ (p_i (x)) : k ]<\infty$, then $M= k[x]/ (p_i (x))  $ is perfect, that is $M^3= M$. The previous step can be achieved by finding a root $\delta$ for $p_i(x)$ in $\mathbb{F}_{q^{\deg(p_i(x))}}$, then computing $ t^{(i)}_0 (\delta )^2$ in $\mathbb{F}_{q^{\deg(p_i(x))}}$, and finding $\beta \in \mathbb{F}_{q^{\deg(p_i(x))}}$ such that $\beta^3 = t^{(i)}_0 (\delta)^2$ in $\mathbb{F}_{q^{\deg(p_i(x))}}$. Then $m(x) = \beta + p_i(x)$ will satisfy the congruence above. }

\item[{\bf Step 2a}]  
We set 
\begin{itemize} 
\item[$\cdot$] $$ z = \frac{1}{a} \left( y -  \big( \frac{ m(x)}{ p_i (x)^{2 \lambda_i}}\big) \right) $$
\item[$\cdot$] $$ b = \frac{( a^2 + \big( \frac{ m(x)}{ p_i (x)^{2 \lambda_i}}\big) ^3 + a\big( \frac{ m(x)}{ p_i (x)^{2 \lambda_i}}\big) \big)   ^2}{ a^3}$$ 
\end{itemize}

Using the {\sc Factorization algorithm}, we write 
$$b = \frac{ g(x)}{ p_i(x)^{\alpha'_i} q(x)}$$ 
with $q(x)$ polynomial and $g(x)$ polynomial, $(p_i(x) ,  q(x))=1$ and $(g(x), p_i(x))=1$, $(g(x), q(x))=1$ and $\alpha_i'$ an integer.
\begin{enumerate} 
\item If $\alpha'_i \geq 0$; Change $i$ in {\bf STEP 2} with $a=b$ and $y=z$, if there are no more $i$ to work with,  taking $a=b$ and $y=z$ exit to {\bf STEP 3}. 
\item If $\alpha'_i < 0$ with $r \nmid \alpha'_i $; Change $i$ in {\bf STEP 2} with $a=b$ and $y=z$, if there are no more $i$ to work with, taking $a=b$ and $y=z$ exit to {\bf STEP 3}.
\item Otherwise, repeat {\bf Step 2} with taking $a=b$ and $y=z$.
\end{enumerate}

{\small {\it Note} that one can write 
$$ a^2 =  \frac{ t_0^{(i)} (x)^2 }{ p_i (x)^{6 \lambda_i }} + t_2(x)$$ 
where 
$$v_{\mathfrak{p}_i} ( t_2 (x) ) > -6\lambda_i $$ 
We get 
$$v_{p_i} ( a^2 + w^3 + a w) \geq \min \{ v_{\mathfrak{p}_i} \left(  \frac{ t_0^{(i)} (x)^2 }{ p_i (x)^{6 \lambda_i }} + \frac{m^3 (x)}{ p_i(x)^{ 6 \lambda_i}}  \right) + v_{\mathfrak{p}_i }( t_2 (x)) + v_{\mathfrak{p}_i } (a w)\}$$ 
 
But, 
$$v_{\mathfrak{p}_i} \left(  \frac{ t_0^{(i)} (x)^2 }{ p_i (x)^{6 \lambda_i }} + \frac{m^3 (x)}{ p_i(x)^{ 6 \lambda_i}}  \right)\geq 1 - 6 \lambda_i > - 6 \lambda_i, $$ 

$$v_{\mathfrak{p}_i }( t_2 (x))  > - 6 \lambda_i, $$ 
and 
$$ v_{\mathfrak{p}_i } (a w)= v_{\mathfrak{p}_i } (a )+ v_{\mathfrak{p}_i } ( w) = -3\lambda_i - 2 \lambda_i > -6 \lambda_i $$ Therefore, 
$$ v_{p_i} ( a^2 + w^3 + a w)> 2 v_{\mathfrak{p}_i} (a) $$ 
and 

$$v_{p_i} (b)= v_{p_i} \left( \frac{( a^2 + w^3 + a w )^2}{ a^3} \right) >  v_{p_i} (  a)$$ 
{\it Note} that the other prime valuations are affected in a way that does not cause any nonnegative valuation to become negative (see \cite[Lemma 1.2]{MWcubic3} to find the argument one can use to prove this). So the process will finish in finite time. 
 }
 
 \item[{\bf STEP 3}] We set
 $$s = \deg(a)$$ 
 \begin{enumerate}
 \item If $s \leq 0$, move to {\bf STEP 4}. 
 \item If $s>0$, $3 \nmid s$, move to {\bf STEP 4}. 
 \item Otherwise 
 \begin{enumerate} 
\item Use the {\sc euclidean algorithm} to write $s = 3d$ with $d$ a natural number,
\item Write 
$$ a = \frac{g(x)}{ h(x)}$$ 
where $g(x)$, $h(x)$ polynomials with $(f(x), g(x))=1$. 
 \item Apply the {\sc Euclidean algorithm} to find $q(x)$ and $r(x)$ polynomial with $ \deg( r(x)) <h(x) $ or $r(x) =0$, and 
 $$ g(x) =h(x)  q(x) + r(x)$$ 
 
 \item Write $q(x) = \alpha x^{3t} + t(x)$ 
 
 {\it {\small Note $\deg(t(x))< 3t$.}} 
 \item Find $\beta$ in $\mathbb{F}_q$ such that $\beta ^3= \alpha^2 $.
 \item Take
 \begin{itemize} 
 \item[$\cdot$] $$z  = \frac{1}{a} \big( y+  \beta x^{2t} \big)$$
\item[$\cdot$] $$b =  \frac{( a^2 - (\beta x^{2t})^3 - a (\beta x^{2t})  )^2}{ a^3}$$ 
 \end{itemize} 
 \item  \begin{itemize} 
 \item[$\cdot$] If $r | \deg(b)$, then redo {\bf STEP 3} with $y=z$ and $a =b$. 
 \item[$\cdot$] otherwise exit to {\bf STEP 4}. 
 \end{itemize} 
{\small {\it Note $\deg (b) \leq 3t-1< 3t$ }. Note also that the other prime valuations are affected in a way that does not cause any non negative valuations to become negative. Hence, the entire process will end in finite time.}
 \end{enumerate} 
 \end{enumerate} 
\item[{\bf STEP 4}] {\bf RETURN} 
\begin{itemize} 
\item[$\cdot$] $z$ in terms of initial $y$ given into the algorithm, 
\item[$\cdot$]  $$b = \frac{ g(x)}{ \prod_{j=1}^s p_{i_j} (x)^{\alpha_{i_j}^t}} $$ 
with $g(x)$ polynomial, $\alpha_{i_j}^t$ natural number with $3 \nmid \alpha_{i_j}^t$, for $j \in \{ 1, \cdots , s\}$.
\end{itemize} 
 \end{enumerate} 

{\small {\it Note} that it is also easy to check if a cubic extension over $\mathbb{F}_q(x)$ with generator $y$ such that its minimal polynomial is $x^3 +ax^2 + a$ is Galois, as one only needs to check if $-a$ is a square in $\mathbb{F}_q(x)$. When it is a square, the extension is Galois, and one finds $b$ such that $-a = b^2$ using for instance the {\sc factorization algorithm}, and then $\frac{y}{a}$ is an Artin-Schreier generator with minimal polynomial $x^3 -x - \frac{1}{a}$.}  \\ \bigskip

\begin{center} 
{\sc Application: finding integral basis for an extension with minimal polynomial of the form $x^3 +ax^2 + a$.} 
\end{center}
The statement used for this Algorithm is done in \cite[Theorem 3]{MadMad}, and finds explicitly an integral basis for any purely cubic extension.\\ \bigskip

\begin{center}
{\bf Algorithm integral basis 3: takes $y$ (generator of $L/K$ with minimal polynomial), $X^3+aX+a^2 $ and $p= 3$}
\end{center}
\begin{enumerate} 
\item Use the {\sc factorization algorithm}, to factorize $a$ as 
$$ \frac{ f(x)}{\prod_{i=1}^m p_i(x)^{\alpha_i}}$$ 
with $p_i(x)$ distinct irreducible polynomials, $f(x)$  polynomial and $(f(x) ,   \prod_{i=1}^m p_i(x)^{\alpha_i}) =1$, $\alpha_i$ natural number and $$f(x) = \prod_{i=1}^u r_i(x)^{\beta_i}$$
where $r_i(x)$ distinct irreducible polynomials and $\beta_i$ natural integer, for $i \in \{ 1, \cdots , u\}$. 

\item Use the {\sc Generalized Artin-Schreier Algorithm} above, it will return 
\begin{itemize} 
\item[$\cdot$] $z$ in terms of $y$. 
\item[$\cdot$]  $$b = \frac{ g(x)}{ \prod_{j=1}^s p_{i_j} (x)^{\alpha_{i_j}^t}} $$ 
with $g(x)$ polynomial, $\alpha_{i_j}^t$ natural number with $3 \nmid \alpha_{i_j}^t$, for $j \in \{ 1, \cdots , s\}$.
\end{itemize} 

 \item Write $b$ into 
$$b= \frac{ \xi_1 \xi_2^2 }{ \prod_{j=1}^s p_{i_j} (x)^{\alpha_{i_j}^t}},$$ 
with $\xi_1, \xi_2  \in \mathbb{F}_q[x]$, $\xi_1$ is square-free, and $(\xi_1 \xi_2, \beta ) = 1$.

\end{enumerate}

{\bf RETURN:} {\bf Integral basis for $L/K$:}
$$\mathfrak{B}=\left\{ \frac{P_2}{\xi_1\xi_2^2} z^2 , \frac{P_1}{\xi_2} z, 1 \right\}$$
 where $$P_k = \prod_{j=1}^s p_{i_j}(x)^{1+ \left\lfloor \frac{k 2\alpha_{i_j}^t}{3}\right\rfloor},$$ 
 for $k=1,2$, where $\left\lfloor \frac{k 2\alpha_{i_j}^t}{3}\right\rfloor$ is the integral part of  $\frac{k 2 \alpha_{i_j}^t}{3}$.

\bibliographystyle{plain}
\raggedright
\bibliography{references}

\begin{thebibliography}{1}

\bibitem{Con}
K.~Conrad.
\newblock Galois groups of cubics and quartics in all characteristics.
\newblock {\em Unpublished note}.

\bibitem{MadMad}
M.~Madan and D.~Madden.
\newblock The exponent of class groups in congruence function fields.
\newblock {\em Acta Arith.}, 32(2):183--205, 1977.

\bibitem{MWcubic}
S.~Marques and K.~Ward.
\newblock A complete classification of cubic function fields over any finite
  field.
\newblock {\em Preprint}, 2017.

\bibitem{MWcubic3}
S.~Marques and K.~Ward.
\newblock Cubic fields: A primer.
\newblock {\em European Journal of Mathematics}, pages 1--20, 2018.

\bibitem{MWcubic4}
S.~Marques and K.~Ward.
\newblock An explicit triangular integral basis for any separable cubic
  extension of a function field.
\newblock {\em European Journal of Mathematics}, pages 1--15, 2018.

\bibitem{Neu}
J.~Neukirch.
\newblock {\em Algebraic Number Theory}.
\newblock Springer, 1999.

\bibitem{Sti}
H.~Stichtenoth.
\newblock {\em Algebraic Function Fields and Codes}.
\newblock Springer, 2009.

\bibitem{Vil}
G.D. Villa-Salvador.
\newblock {\em Topics in the Theory of Algebraic Function Fields}.
\newblock Birkh\"{a}user, 2006.

\end{thebibliography}
\vspace{.5cm}

\end{document}